\documentclass[reqno,a4paper,12pt]{amsart}

\usepackage[all,poly]{xy}
\usepackage{amsfonts,stmaryrd}
\usepackage[mathcal]{eucal}
\usepackage{amssymb}
\usepackage{amsmath}
\usepackage{mathrsfs}
\usepackage{color}
\usepackage[pagebackref,colorlinks]{hyperref}
\usepackage{enumerate}
\usepackage{pifont}

\theoremstyle{plain}
\newtheorem {Lem}{Lemma}

\newtheorem {The}{Theorem}
\newtheorem{oldtheorem}{Theorem}

\newtheorem {Cor}{Corollary}

\newtheorem {Prob}{Problem}

\theoremstyle{remark}

\theoremstyle{definition}

\def\bar{\overline}
\def\map{\longrightarrow}

\def\e{\varepsilon}

\def\GL{\operatorname{GL}}

\def\SL{\operatorname{SL}}

\def\Sp{\operatorname{Sp}}
\def\SO{\operatorname{SO}}

\def\EU{\operatorname{EU}}
\def\EEU{\operatorname{EEU}}
\def\FU{\operatorname{FU}}
\def\CU{\operatorname{CU}}

\def\B{\operatorname{B}}
\def\C{\operatorname{C}}

\def\F{\operatorname{F}}
\def\G{\operatorname{G}}
\def\K{\operatorname{K}}

\def\jdim{\operatorname{j-dim}}

\def\unlhd{\trianglelefteq}

\def\map{\longrightarrow}
\def\pamod#1{\,(\operatorname{mod}{\, #1})\,}

\newcommand{\FormR}{A,\Lambda}

\newcommand{\FidealI}{I,\Gamma} 
\newcommand{\FidealJ}{J,\Delta}
\newcommand{\FidealK}{K,\Omega}
\newcommand{\FidealL}{L, \Theta}

\def\Cent{\operatorname{Cent}}

\def\jdim{\operatorname{j-dim}}
\def\sr{\operatorname{sr}}

\def\map{\longrightarrow}
\def\bar{\overline}
\def\epsilon{\varepsilon}
\def\e{\varepsilon}

\def\Ga{\Gamma}

\newcommand\Label[1]{\label{#1}}

\newcommand{\ep}{\varepsilon}

\newcommand\KH{\operatorname{KH}}
\newcommand\KU{\operatorname{KU}}

\newcommand\StU{\operatorname{StU}}

\def\map{\longrightarrow}

\def\e{\varepsilon}

\def\bar{\overline}

\def\B{\operatorname{B}}
\def\C{\operatorname{C}}

\def\F{\operatorname{F}}
\def\G{\operatorname{G}}
\def\K{\operatorname{K}}

\def\Sp{\operatorname{Sp}}

\def\SO{\operatorname{SO}}

\def\SL{\operatorname{SL}}
\def\GL{\operatorname{GL}}

\def\GU{\operatorname{GU}}

\def\EE{\operatorname{EE}}

\def\sr{\operatorname{sr}}

\long\def\forget#1\forgotten{}

\title[commutators of relative elementary unitary groups]
{commutators of relative and unrelative\\
 elementary unitary groups}

\author{N.~Vavilov}
\address{Department of Mathematics and Computer Science,
St.~Petersburg State University, St.~Petersburg, Russia}
\email{nikolai-vavilov@yandex.ru}
\thanks{The work of the first author was supported by the
Russian Science Foundation grant 17-11-01261.}
\author{Z.~Zhang}
\address{Department of  Mathematics, Beijing Institute
of Technology, Beijing, China}
\email{zuhong@hotmail.com}


\begin{document}

\begin{abstract}
In the present paper, which is an outgrowth of our joint work
with Anthony Bak and Roozbeh Hazrat on unitary commutator 
calculus
\cite{BV3, RNZ1, RNZ4, RNZ5}, we find generators of the mixed
commutator subgroups of relative elementary groups and 
obtain unrelativised versions of commutator 
formulas in the setting of Bak's unitary groups.
It is a direct sequel of our papers \cite{NV18, NZ2, NZ3, NZ6} and \cite{NZ1, NZ4},
where similar results were obtained for $\GL(n,R)$ and
for Chevalley groups over a commutative ring with 1, respectively.
Namely, let $(A,\Lambda)$ be any form ring and $n\ge 3$.
We consider Bak's hyperbolic unitary group $\GU(2n,A,\Lambda)$.
Further, let $(I,\Gamma)$ be a form ideal of $(A,\Lambda)$.
One can associate with $(I,\Gamma)$ the corresponding
elementary subgroup $\FU(2n,I,\Gamma)$ and the
relative elementary subgroup $\EU(2n,I,\Gamma)$
of $\GU(2n,A,\Lambda)$. Let $(J,\Delta)$ be another form
ideal of $(A,\Lambda)$. 
In the present paper we prove an unexpected result that 
the non-obvious type of generators for
$\big[\EU(2n,I,\Gamma),\EU(2n,J,\Delta)\big]$,
as constructed in our previous papers with Hazrat, are 
redundant and can be expressed as products of the obvious
generators, the elementary conjugates
$Z_{ij}(ab,c)=T_{ji}(c)T_{ij}(ab)T_{ji}(-c)$ and $Z_{ij}(ba,c)$,
and the
elementary commutators $Y_{ij}(a,b)=[T_{ji}(a),T_{ij}(b)]$,
where $a\in(\FidealI)$, $b\in(\FidealJ)$, $c\in(\FormR)$.
It follows that 
$\big[\FU(2n,I,\Gamma),\FU(2n,J,\Delta)\big]=
\big[\EU(2n,I,\Gamma),\EU(2n,J,\Delta)\big]$.
In fact, we establish much more precise generation results. 
In particular, even the elementary commutators $Y_{ij}(a,b)$ 
should be 
taken for one long root position and one short root position.  
Moreover, $Y_{ij}(a,b)$ are central modulo
$\EU(2n,(\FidealI)\circ(\FidealJ))$ and behave as symbols.
This allows us to generalise and unify many previous 
results,including the multiple elementary commutator
formula, and dramatically simplify their proofs.
\end{abstract}

\maketitle

\maketitle
\hangindent 5.5cm\hangafter=0\noindent
To our 
dear friend Mohammad Reza Darafsheh,\\
\phantom{xxxxxxxxx} with affection and admiration
\par\hangindent 5.5cm\hangafter=0\noindent
\bigskip

\section*{Introduction}

In a series of our joint papers with Anthony Bak and 
Roozbeh Hazrat \cite{BV3, RNZ1, RNZ4, RNZ5} we studied 
commutator formulas in Bak's unitary groups. In the present 
paper we generalise, refine and strengthen some of the main 
results of these works. Namely, we discover that
the set of generators for the mixed commutator subgroup
of relative elementary unitary groups
listed in these papers can be substantially reduced and remove
all commutativity conditions therein\footnote{In particular,
this solves \cite{yoga-2}, Problem~1 and \cite{RNZ4}, Problem 1.}. 
This allows us to prove unexpected unrelative versions of the 
commutator formulas, generalise multiple elementary commutator formulas, and more. These results both improve a great number of previous results, and path the way to several new unexpected applications.

Morally, the present paper is a direct sequel our papers
 \cite{NV18, NZ2, NZ3, NZ6} and \cite{NZ1, NZ4},
where the same was done for $\GL(n,R)$ and for Chevalley groups
over a commutative ring with 1, respectively. There, the proofs
heavily relied on our previous works, in particular on
\cite{ASNV, NVAS, NVAS2, RHZZ1, RHZZ2} for $\GL(n,R)$ and
on \cite{RNZ2, RNZ3} for Chevalley groups. Similarly, the present
paper heavily hinges on the results of \cite{BV3, RNZ1, RNZ4, 
RNZ5}.


\subsection{The prior state of art} To enunciate the main results 
of the present papers, let us briefly recall the notation, which will
be reviewed in somewhat more detail in \S\S~1--4.
Let $(A,\Lambda)$ be a form ring,
$n\ge 3$, and let $\GU(2n,A,\Lambda)$ be the hyperbolic
Bak’s unitary group.
Below, $\EU(2n, A,\Lambda)$ denotes the [absolute] elementary unitary
group, generated by the elementary root unipotents.

As usual, for a form ideal $(I,\Gamma)$ of the
form ring $(A,\Lambda)$ we denote by
 $\FU(2n,I,\Gamma)$ the unrelative
elementary subgroup of level $(I,\Gamma)$, and by
$\EU(2n, I,\Gamma)$
the relative elementary subgroup of level
$(I,\Gamma)$. By definition,
$\EU(2n, I,\Gamma)$ is the normal closure of  $\FU(2n,I,\Gamma)$
in $\EU(2n, A,\Lambda)$. Further, $\GU(2n, I,\Gamma)$
and $\CU(2n, I,\Gamma)$
denote the principal congruence subgroup and the
full congruence subgroup of level $(I,\Gamma)$, respectively.

Let us recapitulate two principal results of our joint papers with
Roozbeh Hazrat, \cite{RNZ1, RNZ4, RNZ5}. 
The first one is the birelative standard commutator
formula, \cite{RNZ1}, Theorems~1 and 2. It is a very broad generalisation
of the commutator formulas for unitary groups, previously established
by Anthony Bak, the first author, Leonid Vaserstein, Hong You, 
Günter Habdank, and others, see, for instance \cite{B1, B2, BV3, VY, Ha1, Ha2, BP}.

\begin{oldtheorem}\label{standard}
Let $R$ be a commutative ring, $(A,\Lambda)$ be a form ring
such that $A$ is a quasi-finite $R$-algebra. Further, let $(\FidealI)$
and $(\FidealJ)$ be two form ideals of the form ring $(\FormR)$
and let $n\ge 3$. Then the following commutator identity holds
$$   [\GU(2n,\FidealI), \EU(2n, \FidealJ)] =
[\EU(2n,\FidealI), \EU(2n, \FidealJ)]. $$
\noindent
When $A$ is itself commutative, one even has
$$   [\CU(2n,\FidealI),\EU(2n,\FidealJ)] =
[\EU(2n,\FidealI),\EU(2n,\FidealJ)]. $$
\end{oldtheorem}

Another crucial result is description of a generating set for the
mixed commutator subgroup 
$[\EU(2n,\FidealI), \EU(2n,\FidealJ)]$ {\it as a group\/}, similar
to the familiar generating set for relative elementary subgroups,
see \cite{BV3}, Proposition 5.1 (compare Lemma 3 below). 
\par
Recall that we denote by $T_{ij}(a)$ elementary unitary
transvections. They come in two denominations, those of 
{\it short root type\/}, when $i\neq\pm j$, and those of 
{\it long root type\/}, when $i=-j$. The corresponding root
subgroups are then parametrised by the ring $A$ itself and by
the form parameter $\Lambda$, respectively. To simplify
notation in the relative case, we introduce the following convention. 
For a form ideal $(\FidealI)$ we write $a\in(\FidealI)$ to denote 
that $a\in I$ if $i\neq\pm j$, and 
$a\in\lambda^{-(\e(i)+1)/2}\Gamma$ if $i=-j$. Clearly,
$a\in(\FidealI)$ means precisely that $T_{ij}(a)\in\EU(2n,\FidealI)$,
see \S\S~3,4 for details. 
\par
Further, we consider
the elementary conjugates $Z_{ij}(a,c)$ and the elementary commutators $Y_{ij}(a,b)$, which are defined as follows:
$$ Z_{ij}(a,c)=T_{ji}(c)T_{ij}(a)T_{ji}(-c),
\qquad Y_{ij}(a,b)=[T_{ji}(a),T_{ij}(b)], $$
\par
The following result in a slightly weaker
form was stated as Theorem~9 of \cite{RNZ5}, and in
precisely this form as Theorem 3B of \cite{RNZ4}. Observe
that there its proof depended on Theorem A, and thus 
ultimately, on localisation methods.

\begin{oldtheorem}\label{oldgenerators}
Let $R$ be a commutative ring, $(A,\Lambda)$ be a form ring
such that $A$ is a quasi-finite $R$-algebra.
 Let $(\FidealI)$ and $(\FidealJ)$ be two form ideals of the form ring $(\FormR)$ and let $n\ge 3$.
The relative commutator subgroup  $[\EU(2n,\FidealI), \EU(2n, \FidealJ)]$ is generated by the elements of the following three 
types
\par\smallskip
$\bullet$ $Z_{ij}(ab,c)$ and $Z_{ij}(ba,c)$,
\par\smallskip
$\bullet$ $Y_{ij}(a,b)$, 
\par\smallskip
$\bullet$ $[T_{ij}(a), Z_{ij}(b,c)]$, 
\par\smallskip\noindent
where in all cases $a\in(\FidealI)$, $b\in(\FidealJ)$ and
$c\in(\FormR)$  
\end{oldtheorem}

\subsection{Statement of the principal result}
The technical core of the present paper are Lemmas 6--12 
that we prove in \S\S~5--8. Together they imply
that the above Theorem B can be {\it drastically\/} generalised
and improved, as follows:
\par\smallskip
$\bullet$ We can lift the commutativity condition. 
\par\smallskip
$\bullet$ The third type of generators are redundant. 
\par\smallskip
$\bullet$ The second type of generators can be
restricted to one long and one short root (and are subject 
to further relations, to be stated below). 
\par\smallskip
The following result is the pinnacle of the present paper, 
other results are either preparation to its proof, or its easy 
corollaries. For the general linear group $\GL(n,R)$ it was 
established in \cite{NZ2}, Theorem 1. For Chevalley groups
$G(\Phi,R)$ over commutative rings
--- and thus, in particular, for the usual symplectic group
$\Sp(2n,R)$ and the split orthogonal group $\SO(2n,R)$ ---
it is essentially a conjunction of \cite{NZ1}, Theorem 1.2, and \cite{NZ4}, Theorem 1. However, as explained below, in these
special cases one can say somewhat more.

\begin{The}\label{generators}
Let $(\FormR)$ be any associative form ring, let $(\FidealI)$ 
and $(\FidealJ)$ be two form ideals of the form ring 
$(\FormR)$ and let $n\ge 3$. Then the relative commutator 
subgroup $[\EU(2n,\FidealI), \EU(2n, \FidealJ)]$ is generated 
by the elements of the following two types
\par\smallskip
$\bullet$ $Z_{ij}(ab,c)$ and $Z_{ij}(ba,c)$,
\par\smallskip
$\bullet$ $Y_{ij}(a,b)$,
\par\smallskip\noindent
where in all cases $a\in(\FidealI)$, $b\in(\FidealJ)$ and
$c\in(\FormR)$. Moreover, for the second type of generators 
it suffices to take one pair $(h,k)$, $h\neq\pm k$, and one 
pair $(h,-h)$. 
\end{The}

The difference with Chevalley groups is that now we have 
to throw in elementary commutators for {\it two\/} roots,
one long root and one short root. For Chevalley groups, one 
{\it long\/} root would suffice. Conversely, when 2 is invertible 
for types $\B_l,\C_l,\F_4$ and 3 is invertible for type $\G_2$,
one {\it short\/} toot would suffice. For unitary groups, 
modulo $\EU(2n,(\FidealI)\circ(\FidealJ))$ we can still establish 
a cognate relation between short root type elementary 
commutators and long root type elementary commutators, 
Lemma~\ref{long-short}. However, unlike Chevalley groups, for unitary 
groups the elements of long root subgroups are parametrised by
the form parameter $\Lambda$, whereas the elements of short 
root subgroups are parametrised by the ring $A$ itself. This
means that now we could dispose of {\it some\/} short type 
elementary commutators, yet not all of them. In the opposite 
direction, the long type elementary commutators, one of whose arguments sits in the corresponding {\it minimal\/} ideal form 
parameter could be discarded --- but not all of them! This can
be done when one of the form parameters is either minimal,
or as large as possible --- not just maximal! --- see \S~9.
\par
Observe that the proof of this theorem consists of two
independent parts. The possibility to express the third type 
of generators as products of elementary conjugates and elementary commutators in $[\EU(2n,\FidealI),\EU(2n,\FidealJ)]$ will be called 
the {\it first claim\/} of Theorem~\ref{generators}. The much more arduous bid 
that modulo $\EU(2n,(\FidealI)\circ(\FidealJ))$ all elementary 
commutators can be expressed in terms of such commutators 
in one short and one long positions, will be called the {\it second 
claim\/} of Theorem~\ref{generators}.

Let us mention another important trait. The published
proofs of Theorem B heavily depended on some version of
Theorem A, and thus, ultimately, on localisation. The proof
of Theorem~\ref{generators} given below in \S\S~5--7 is purely 
{\it elementary\/}\footnote{In the technical sense that it does 
not invoke anything apart from the usual Steinberg relations.}
and thus works already at the level of {\it unitary Steinberg groups\/},
see \cite{B1, B2, lavrenov}. The only reason why we do not 
state our results in this generality is to skip discussion 
of {\it relative unitary Steinberg groups\/}. The details and technical facts are not readily available in the literature, and would noticeably increase the length of the present paper.


\subsection{Unrelativisation}
Since both remaining types of generators listed in Theorem~\ref{generators}  
already belong 
to the mixed commutator of the {\it unrelative\/} elementary subgroups
$[\FU(2n,\FidealI), \FU(2n, \FidealJ)]$, we get the following amazing
equality. Morally, it shows that the commutator of {\it relative\/} elementary subgroups $[\EU(2n,\FidealI),\EU(2n, \FidealJ)]$ is
smaller, than one expects. Observe that it only depends on the
[relatively] easy first claim of Theorem~\ref{generators} whose proof is 
completed already in \S~5. For $\GL(n,R)$ the corresponding result 
is \cite{NV18}, Theorem 2 (for commutative rings, with a completely 
different proof), and \cite{NZ2}, Theorem 1 (for arbitrary associative rings). For $\Sp(2n,R)$ and $\SO(2n,R)$ it is a special case of 
\cite{NZ1}, Theorem 1.2.

\begin{The}\label{equality2}
Let $(\FormR)$ be any associative form ring, let $(\FidealI)$ 
and $(\FidealJ)$ be two form ideals of the form ring 
$(\FormR)$ and let $n\ge 3$.  Then the mixed commutator 
subgroup $[\FU(2n,\FidealI), \FU(2n, \FidealJ)]$ is normal in 
$\EU(2n,\FormR)$. Furthermore, we have the following
commutator identity
$$   [\FU(2n,\FidealI), \FU(2n, \FidealJ)] =
[\EU(2n,\FidealI), \EU(2n, \FidealJ)]. $$

\end{The}

In particular, in conjunction with Theorem A this shows 
that the birelative standard commutator formula also holds 
in the following unrelativised form. Again, for $\GL(n,R)$
this is \cite{NV18}, Theorem 1 and \cite{NZ2}, Theorem 3,
whereas for Chevalley groups it is \cite{NZ1}, Theorem 1.3.

\begin{The}\label{unrelative}
Let $R$ be a commutative ring, $(A,\Lambda)$ be a form ring
such that $A$ is a quasi-finite $R$-algebra. Further, let $(\FidealI)$
and $(\FidealJ)$ be two form ideals of the form ring $(\FormR)$
and let $n\ge 3$.  Then we have a unrelative commutator identity
$$   [\GU(2n,\FidealI),\EU(2n,\FidealJ)] =
[\FU(2n,\FidealI),\FU(2n,\FidealJ)]. $$
\noindent
When $A$ is itself commutative, one even has
$$   [\CU(2n,\FidealI),\EU(2n,\FidealJ)] =
[\FU(2n,\FidealI),\FU(2n,\FidealJ)]. $$
\end{The}

The following result is a unitary analogue of the unrelative 
normality theorem proven for $\GL(n,R)$ by Bogdan Nica 
and ourselves, see \cite{Nica, NV18, NZ3}. It is an immediate 
corollary of
our Theorem~\ref{unrelative}, if you set there $(\FidealI)=(\FidealJ)$.

\begin{The}\label{T:4}
Let $R$ be a commutative ring, $(A,\Lambda)$ be a form ring
such that $A$ is a quasi-finite $R$-algebra. Further, let $(\FidealI)$
be a form ideals of the form ring $(\FormR)$ and let $n\ge 3$. 
Then $\FU(2n,\FidealI)$ is normal in $\GU(n,\FidealI)$.
\end{The}


\subsection{Elementary commutators}
The proof of the second claim of Theorem~\ref{generators} is the gist of the
present paper, and proceeds as follows. First, in \S~6 we prove
that the elementary commutators $Y_{ij}(a,b)$ are central in
the absolute elementary group modulo 
$\EU(2n,(\FidealI)\circ(\FidealJ))$. 
Recall that here 
$$ (\FidealI)\circ(\FidealJ)=\big(IJ+JI,{}^J\Gamma+{}^I\Delta+\Gamma_{\min}(IJ+JI)\big) $$
\noindent
denotes the symmetrised product of form ideals, see \S~2 for details.
\par

Since by that time we 
already know that together with $\EU(2n,(\FidealI)\circ(\FidealJ))$
these commutators generate $[\FU(2n,\FidealI),\FU(2n, \FidealJ)]$,
this result can be stated as follows. For $\GL(n,R)$ and Chevalley 
groups this is \cite{NZ2}, Theorem 2, and \cite{NZ4}, Theorem 2,
respectively.

\begin{The}\label{equality}
Let $(\FormR)$ be any associative form ring, let $(\FidealI)$ 
and $(\FidealJ)$ be two form ideals of the form ring 
$(\FormR)$ and let $n\ge 3$.  Then 
$[\FU(2n,\FidealI),\FU(2n, \FidealJ)]$ is central in 
$\EU(2n,\FormR)$ modulo $\EU(2n,(\FidealI)\circ(\FidealJ))$.
\end{The}

In other words,
$$ \Big[\big[\FU(2n,\FidealI),\FU(2n,\FidealJ)\big],
\EU(2n,\FormR)\Big]\le \EU(2n,(\FidealI)\circ(\FidealJ)). $$
\noindent
In particular, it implies that the quotient 
$$ \big[\FU(2n,\FidealI),\FU(2n,\FidealJ)\big]/ 
\EU(2n,(\FidealI)\circ(\FidealJ)) $$
\noindent
is itself abelian. This readily implies additivity of the elementary 
commutator with respect to its arguments, and other similar
useful properties, collected in Theorem~\ref{symb}, that are employed 
in the proofs of subsequent results.

However, the focal point of the present paper is \S~7, 
where we prove that modulo $\EU(2n,(\FidealI)\circ(\FidealJ))$
all elementary commutators of the same root type are equivalent.
Moreover, for the short root type they are balanced with respect 
to the factors from $R$, both on the left and on the right. 
For the long root type the balancing property is more 
complicated, and only holds for the quadratic (=Jordan)
multiplication. In the case of the usual symplectic group,
where $A$ is a commutative ring with trivial involution,
it corresponds to the multiplication by squares, see \cite{NZ4},
Theorem 5.

\begin{The}\label{The:Op-2.1}
Let $(\FormR)$ be an associative form ring with $1$, $n\ge 3$, 
and let $(\FidealI)$, $(\FidealJ)$ be form ideals of $(\FormR)$. 
\par\smallskip
$\bullet$
Then for any  $ i\neq \pm j$, any $h\ne\pm l$  with 
$h,l \neq \pm i, \pm j$, and $a\in I$, $b\in J$, $c,d\in A$, 
the elementary commutator
$$ Y_{ij}(cad,b)\equiv Y_{hl}(a,dbc)
\pamod{\EU(2n,(\FidealI)\circ(\FidealJ))}.  $$ 
\par\smallskip
$\bullet$ Then for any  $ -n\le i\le n$, any
$-n\le k\le n$, and $a\in \lambda^{-(\epsilon(i)+1)/2}\Gamma$, $b\in \lambda^{(\epsilon(i)-1)/2}\Delta$ $c\in A$, the elementary
commutator
$$ Y_{i,-i}(ca\bar c,b)\equiv Y_{k,-k}(\lambda^{(\epsilon(i)-\epsilon(k))/2}a,-\lambda^{(\epsilon(k)-\epsilon(i))/2}\bar c  bc)\pamod{\EU(2n,(\FidealI)\circ(\FidealJ))}.  $$ 
\end{The}

The calculation behind these congruences is the highlight of the
whole theory. Inherently, it is just a birelative incarnation of a
classical calculation that appeared dozens of times in the 
algebraic K-theory and the theory of algebraic groups since
mid 60-ies, see \S~12 for a terse historical medley.


\subsection{Further corollaries}
As another illustration of the power of Theorem~\ref{generators}, we show that
it allows to [almost completely] lift commutativity conditions in 
some of the principal results of \cite{RNZ1,RNZ4,RNZ5}.

\par
Under the additional assumptions such as quasi-finiteness
the following result for any $n\ge 3$ is \cite{RNZ5}, Theorem 7.
From Theorem~\ref{generators} we can derive that for $n\ge 4$ a similar result 
holds for arbitrary associative form rings. For $\GL(n,R)$ 
such generalisation was already obtained in  \cite{NZ2}. We 
believe this could be also done for $n=3$, see Problem 3, but in
that case it would require formidable calculations.

\begin{The}\label{T:7}
Let $(\FormR$) be any associative form ring with $1$, let $n\ge 4$, 
and let $(I_i,\Gamma_i)\unlhd R$, $i=1,\ldots,m$,  be form ideals 
of $(\FormR)$. Consider an arbitrary arrangement of brackets\/ 
$\llbracket \ldots\rrbracket$
with the cut point\/ $s$. Then one has
\begin{multline*}
\Big\llbracket \EU(2n,I_1,\Gamma_1),\EU(2n,I_2,\Gamma_2),\ldots,\EU(2n,I_m,\Gamma_m)\Big\rrbracket=\\
\Big[\EU(2n,(I_1,\Gamma_1)\circ\ldots\circ( I_s,\Gamma_s)),\EU(2n,(I_{s+1},\Gamma_{s+1})\circ\ldots\circ (I_m,\Gamma_m)\Big], 
\end{multline*}
\noindent
where the bracketing of symmetrised products on the right hand side coincides with the bracketing of the commutators on the left hand side.
\end{The}

Under the additional assumption that the absolute standard 
commutator formulae are satisfied, the following result 
is \cite{RNZ1}, Theorem 3. As we know from \cite{BV3, RH,
RH2, RNZ1}, this condition is satisfied for quasi-finite rings.
But from the work of Victor Gerasimov \cite{Gerasimov} it follows 
that 
some commutativity or finiteness assumptions are necessary
for the standard commutator formulae to hold.
Now, we are in a position to prove the following result for 
{\it arbitrary\/} associative form rings.

\begin{The}\label{T:8}
Let\/ $(\FormR)$ be any associative form ring and $n\ge 3$.
Then for any two comaximal form ideals\/ $(I,\Gamma)$ and 
$(J,\Delta)$ of the form ring $(R,\Lambda)$, $I+J=A$, one has 
the following equality
$$ [\EU(2n,I,\Gamma),\EU(2n,J,\Delta)]=
\EU(2n,(\FidealI)\circ(\FidealJ)). $$
\end{The}

Another bizarre corollary of Theorem~\ref{generators} is surjective 
stability of the quotients 
$$ [\FU(2n,\FidealI),\FU(2n,\FidealJ)]/
\EU(2n,(\FidealI)\circ(\FidealJ)), $$
\noindent 
again for {\it arbitrary\/} associative form rings, without any 
stability conditions, or commutativity conditions.
This is a typical result in the style of Bak's paradigm 
``stability results without stability conditions'', see \cite{Bak}
and also \cite{RH, RH2, RN1, RN, BHV}.

\begin{The}\label{T:9}
Let $(\FormR)$ be any associative form ring, let $(\FidealI)$ 
and $(\FidealJ)$ be two form ideals of the form ring 
$(\FormR)$ and let $n\ge 3$. Then the stability map 
\begin{multline*}
[\FU(2n,\FidealI), \FU(2n, \FidealJ)]/
\EU(2n,(\FidealI)\circ(\FidealJ))\map\\
[\FU(2(n+1),\FidealI), \FU(2(n+1), \FidealJ)]/
\EU(2(n+1),(\FidealI)\circ(\FidealJ)) 
\end{multline*}
\noindent
is surjective.
\end{The}

Indeed, in view of Theorems~\ref{generators} and \ref{equality} as a normal subgroup of 
$\EU(2n,\FormR)$ the group $[\EU(2n,\FidealI),\EU(2n,\FidealJ)]$ 
is generated by $[\EU(6,\FidealI),\EU(6,\FidealJ)]$. An explicit
calculation of these quotients presents itself as a natural next
step. However, so far we were unable to resolve it, apart from some special cases, see a discussion in \S~12.
\par

\subsection{Organisation of the paper.}
The rest of the paper is devoted to the proof of these results.
In \S\S~1--4 we recall the
necessary definitions and collect requisite preliminary results.
The next four sections \S\S~5--8 are the technical core of 
the paper. Namely, in \S~5 we prove Theorem~\ref{equality} and derive
first corollaries thereof.  In \S~6 we reduce the set of generators 
of $[\EU(2n,\FidealI),\EU(2n,\FidealJ)]$
to the first two types. In \S~7 we prove Theorem~\ref{The:Op-2.1} and then in
\S~8 establish another cognate result, relating {\it some\/} 
elementary commutators of short root type with {\it some\/} 
elementary commutators of long root type.
This finishes the proof of Theorem~\ref{generators} and its corollaries, and, 
in particular, also of Theorems \ref{equality2}--\ref{T:4}
In \S~9  we establish the special cases of Theorem~\ref{T:7}
pertaining to triple and quadruple commutators, and then in
\S~10 derive
Theorem~\ref{T:7} itself by an easy induction. In \S~11 we 
derive Theorem~\ref{T:8} and yet another corollary of our main 
results. Finally, in \S~12 we describe the general context, 
briefly review recent related publications and 
state several further related open problems.


\section{Notation}

Here we recall some basic notation that will be used throughout
the present paper.

\subsection{General linear group}\Label{general}
Let, as above, $A$ be an associative ring with 1. For natural $m,n$
we denote by $M(m,n,A)$ the additive group of $m\times n$ matrices
with entries in $A$. In particular $M(m,A)=M(m,m,A)$ is the ring of
matrices of degree $m$ over $A$. For a matrix $x\in M(m,n,A)$ we
denote by $x_{ij}$, $1\le i\le m$, $1\le j\le n$, its entry in the
position $(i,j)$. Let $e$ be the identity matrix and $e_{ij}$,
$1\le i,j\le m$, be a standard matrix unit, i.e.\ the matrix which has
1 in the position $(i,j)$ and zeros elsewhere.
\par
As usual, $\GL(m,A)=M(m,A)^*$ denotes the general linear group
of degree $m$ over $A$. The group $\GL(m,A)$ acts on the free right
$A$-module $V\cong A^{m}$ of rank $m$. Fix a base $e_1,\ldots,e_{m}$
of the module $V$. We may think of elements $v\in V$ as columns with
components in $A$. In particular, $e_i$ is the column whose $i$-th
coordinate is 1, while all other coordinates are zeros.
\par
Actually, in the present paper we are only interested in the case,
when $m=2n$ is even. We usually number the base
as follows: $e_1,\ldots,e_n,e_{-n},\ldots,e_{-1}$. All other
occurring geometric objects will be numbered accordingly. Thus,
we write $v=(v_1,\ldots,v_n,v_{-n},\ldots,v_{-1})^t$, where $v_i\in A$,
for vectors in $V\cong A^{2n}$.
\par
The set of indices will be always ordered in conformity with this convention,
$\Omega=\{1,\ldots,n,-n,\ldots,-1\}$. Clearly, $\Omega=\Omega^+\sqcup\Omega^-$,
where $\Omega^+=\{1,\ldots,n\}$ and $\Omega^-=\{-n,\ldots,-1\}$. For an
element $i\in\Omega$ we denote by $\e(i)$ the sign of $\Omega$, i.e.\
$\e(i)=+1$ if $i\in\Omega^+$, and $\e(i)=-1$ if $i\in\Omega^-$.


\subsection{Commutators}\Label{sub:1.4}
Let $G$ be a group. For any $x,y\in G$, ${}^xy=xyx^{-1}$ and $y^x=x^{-1}yx$
denote the left conjugate and the right conjugate of $y$ by $x$,
respectively. As usual, $[x,y]=xyx^{-1}y^{-1}$ denotes the
left-normed commutator of $x$ and $y$. Throughout the present paper
we repeatedly use the following commutator identities:
\begin{itemize}
\item[(C1)] $[x,yz]=[x,y]\cdot {}^y[x,z]$,
\smallskip
\item[(C$1^+$)]
An easy induction, using identity (C1), shows that
$$\Bigg[x,\prod_{i=1}^k u_i\Bigg]=
\prod_{i=1}^k {}^{\prod_{j=1}^{i-1}u_j}[x,u_{i}], $$
\noindent
where by convention $\prod_{j=1}^0 u_j=1$,
\item[(C2)] $[xy,z]={}^x[y,z]\cdot [x,z]$,
\smallskip
\item[(C$2^+$)]
As in (C$1^+$), we have
$$\Bigg[\prod_{i=1}^k u_i,x\Bigg]=
\prod_{i=1}^k {}^{\prod_{j=1}^{k-i}u_j}[u_{k-i+1},x], $$
\smallskip
\item[(C3)]
${}^{x}\big[[x^{-1},y],z\big]\cdot {}^{z}\big[[z^{-1},x],y\big]\cdot
{}^{y}\big[[y^{-1},z],x\big]=1$,
\smallskip
\item[(C4)] $[x,{}^yz]={}^y[{}^{y^{-1}}x,z]$,
\smallskip
\item[(C5)] $[{}^yx,z]={}^{y}[x,{}^{y^{-1}}z]$,
\smallskip
\item[(C6)] If $H$ and $K$ are subgroups of $G$, then $[H,K]=[K,H]$,
\end{itemize}
Especially important is (C3), the celebrated {\it Hall--Witt
identity\/}. Sometimes it is used in the following form, known as
the {\it three subgroup lemma\/}.
\begin{Lem}{\label{HW1}}
Let\/ $F,H,L\trianglelefteq G$ be three normal subgroups
of\/ $G$. Then
$$ \big[[F,H],L\big]\le \big [[F,L],H\big ]\cdot \big [F,[H,L]\big ]. $$
\end{Lem}


\section{Form rings and form ideal}\label{sec2}

The notion of $\Lambda$-quadratic forms, quadratic modules and generalised
unitary groups over a form ring $(A,\Lambda)$ were introduced by Anthony
Bak in his Thesis, see \cite{B1, B2}. In this section, and the next one, we
{\it very briefly\/} review the most fundamental notation and results
that will be constantly used in the sequel. We refer to
\cite{bass73, B2, HO, knus, BV3, tang, RH, RH2, petrov1, 
RNZ1, RNZ4, RNZ5, lavrenov} for details, proofs, and
further references. In the final section we mention some further related
recent works, and some generalisations.


\subsection{Form rings}\label{form algebra}
Let $R$ be a commutative ring with $1$, and $A$ be an (not necessarily
commutative) $R$-algebra. An involution, denoted by $\bar{\phantom{a}}$,
is an
anti-homomorphism of $A$ of order $2$. Namely, for $a,b\in A$,
one has 
$$ \overline{a+b}=\bar a+\bar b,\qquad
\overline{ab}=\bar b\,\bar a,\qquad \bar{\bar a}=a. $$
\par\noindent
Fix an element $\lambda\in\Cent(A)$ such that $\lambda\bar\lambda=1$. One may
define two additive subgroups of $A$ as follows:
$$ \Lambda_{\min}=\{c-\lambda\bar c\mid c\in A\}, \qquad
\Lambda_{\max}=\{c\in A\mid c=-\lambda\bar c\}. $$
\noindent
A {\em form parameter} $\Lambda$ is an additive subgroup 
of $A$ such that
\begin{itemize}
\item[(1)] $\Lambda_{\min}\subseteq\Lambda\subseteq\Lambda_{\max}$,
\smallskip
\item[(2)] $c\,\Lambda\,\bar c\subseteq\Lambda$ for all $c\in A$.
\end{itemize}
The pair $(A,\Lambda)$ is called a {\em form ring}.


\subsection{Form ideals}\label{form ideals}
Let $I\unlhd A$ be a two-sided ideal of $A$. We assume $I$ to be
involution invariant, i.~e.~such that $\bar I=I$. Set
$$ \Gamma_{\max}(I)=I\cap \Lambda, \qquad
\Gamma_{\min}(I)=\{a-\lambda\bar a\mid a\in I\}+
\langle ac\bar a\mid a\in I, c\in\Lambda\rangle. $$
\noindent
A {\em relative form parameter} $\Gamma$ in $(\FormR)$ of level $I$ is an
additive group of $I$ such that
\begin{itemize}
\item[(1)] $\Gamma_{\min}(I)\subseteq \Gamma \subseteq\Gamma_{\max}(I)$,
\smallskip
\item[(2)] $c\,\Gamma\,\bar c\subseteq \Gamma$ for all $c\in A$.
\end{itemize}
The pair $(I,\Gamma)$ is called a {\em form ideal}.
\par
In the level calculations we will use sums and products of form
ideals. Let $(I,\Gamma)$ and $(J,\Delta)$ be two form ideals. Their sum
is artlessly defined as $(I+J,\Gamma+\Delta)$, it is immediate to verify
that this is indeed a form ideal.
\par
Guided by analogy, one is tempted to set
$(I,\Gamma)(J,\Delta)=(IJ,\Gamma\Delta)$. However, it is considerably
harder to correctly define the product of two relative form parameters.
The papers \cite{Ha1,Ha2,RH,RH2} introduce the following definition
$$ \Gamma\Delta=\Gamma_{\min}(IJ)+{}^J\Gamma+{}^I\Delta, $$
\noindent
where
$$ {}^J\Gamma=\big\langle b\,\Gamma\,\bar b\mid b\in J\big\rangle,\qquad
{}^I\Delta=\big\langle a\,\Delta\,\bar a\mid a\in I\big\rangle. $$
\noindent
One can verify that this is indeed a relative form parameter of level $IJ$
if $IJ=JI$.
\par
However, in the present paper we do not wish to impose any such
commutativity assumptions. Thus, we are forced to consider the
symmetrised products
$$ I\circ J=IJ+JI,\qquad
\Gamma\circ\Delta=\Gamma_{\min}(IJ+JI)+{}^J\Gamma+{}^I\Delta\big. $$
\noindent
The notation $\Gamma\circ\Delta$ -- as also $\Gamma\Delta$ is
slightly misleading, since in fact it depends on $I$ and $J$, not
just on $\Gamma$ and $\Delta$. Thus, strictly speaking, one should
speak of the symmetrised products of {\it form ideals\/}
$$ (I,\Gamma)\circ (J,\Delta)=
\big(IJ+JI,\Gamma_{\min}(IJ+JI)+{}^J\Gamma+{}^I\Delta\big). $$
\noindent
Clearly, in the above notation one has
$$ (I,\Gamma)\circ (J,\Delta)=(I,\Gamma)(J,\Delta)+(J,\Delta)(I,\Gamma). $$


\section{Unitary groups}\label{sec3}

In the present section we recall basic notation and facts related to
Bak's generalised unitary groups.


\subsection{Unitary group}\Label{unitary} 
For a form ring $(\FormR)$, one considers the
{\it hyperbolic unitary group\/} $\GU(2n,\FormR)$, see~\cite[\S2]{BV3}.
This group is defined as follows:
\par
One fixes a symmetry $\lambda\in\Cent(A)$, $\lambda\bar\lambda=1$ and
supplies the module $V=A^{2n}$ with the following $\lambda$-hermitian form
$h:V\times V\map A$,
$$ h(u,v)=\bar u_1v_{-1}+\ldots+\bar u_nv_{-n}+
\lambda\bar u_{-n}v_n+\ldots+\lambda\bar u_{-1}v_1. $$
\noindent
and the following $\Lambda$-quadratic form $q:V\map A/\Lambda$,
$$ q(u)=\bar u_1 u_{-1}+\ldots +\bar u_n u_{-n} \mod\Lambda. $$
\noindent
In fact, both forms are engendered by a sesquilinear form $f$,
$$ f(u,v)=\bar u_1 v_{-1}+\ldots +\bar u_n v_{-n}. $$
\noindent
Now, $h=f+\lambda\bar{f}$, where $\bar f(u,v)=\bar{f(v,u)}$, and
$q(v)=f(u,u)\mod\Lambda$.
\par
By definition, the hyperbolic unitary group $\GU(2n,A,\Lambda)$ consists
of all elements from $\GL(V)\cong\GL(2n,A)$ preserving the $\lambda$-hermitian
form $h$ and the $\Lambda$-quadratic form $q$. In other words, $g\in\GL(2n,A)$
belongs to $\GU(2n,A,\Lambda)$ if and only if
$$ h(gu,gv)=h(u,v)\quad\text{and}\quad q(gu)=q(u),\qquad\text{for all}\quad u,v\in V. $$
\par
When the form parameter is neither maximal nor minimal, these groups 
are not algebraic. However, their internal structure is very similar to that
of the usual classical groups. They are also oftentimes called general
quadratic groups, or classical-like groups.

\subsection{Unitary transvections}\Label{elementary1}
{\it Elementary unitary transvections\/} $T_{ij}(\xi)$
correspond to the pairs $i,j\in\Omega$ such that $i\neq j$. They come in
two stocks. Namely, if, moreover, $i\neq-j$, then for any $c\in A$ we set
$$ T_{ij}(c)=e+c e_{ij}-\lambda^{(\e(j)-\e(i))/2}\bar c e_{-j,-i}. $$
\noindent
These elements are also often called {\it elementary short root unipotents\/}.
\noindent
On the other side for $j=-i$ and $c\in\lambda^{-(\e(i)+1)/2}\Lambda$ we set
$$ T_{i,-i}(c)=e+c e_{i,-i}. $$
\noindent
These elements are also often called {\it elementary long root elements\/}.
\par
Note that $\bar\Lambda=\bar\lambda\Lambda$. In fact, for any element
$c\in\Lambda$ one has $\bar c=-\bar\lambda c$ and thus 
$\bar\Lambda$ coincides with the set of products $\bar\lambda c$, where $c\in\Lambda$. This means that in the
above definition $c\in\bar\Lambda$ when $i\in\Omega^+$ and $c\in\Lambda$
when $i\in\Omega^-$.
\par
Subgroups $X_{ij}=\{T_{ij}(c)\mid c\in A\}$, where $i\neq\pm j$, are
called {\it short root subgroups\/}. Clearly, $X_{ij}=X_{-j,-i}$.
Similarly, subgroups $X_{i,-i}=\{T_{ij}(c)\mid
c\in\lambda^{-(\e(i)+1)/2}\Lambda\}$ are called {\it long root subgroups\/}.
\par
The {\it elementary unitary group\/} $\EU(2n,\FormR)$ is generated by
elementary unitary transvections $T_{ij}(c)$, $i\neq\pm j$, $c\in A$,
and $T_{i,-i}(c)$, $c\in\Lambda$, see~\cite[\S3]{BV3}.

\subsection{Steinberg relations}\Label{elementary2}
Elementary unitary transvections $T_{ij}(\xi)$ satisfy the following
{\it elementary relations\/}, also known as {\it Steinberg relations\/}.
These relations will be used throughout this paper.
\par\smallskip
(R1) \ $T_{ij}(c)=T_{-j,-i}(-\lambda^{(\varepsilon(j)-\varepsilon (i))/2}\bar{c})$,
\par\smallskip
(R2) \ $T_{ij}(c)T_{ij}(d)=T_{ij}(c+d)$,
\par\smallskip
(R3) \ $[T_{ij}(c),T_{hk}(d)]=e$, where $h\ne j,-i$ and $k\ne i,-j$,
\par\smallskip
(R4) \ $[T_{ij}(c),T_{jh}(d)]=
T_{ih}(cd)$, where $i,h\ne\pm j$ and $i\ne\pm h$,
\par\smallskip
(R5) \ $[T_{ij}(c),T_{j,-i}(d)]=
T_{i,-i}(cd-\lambda^{-\varepsilon(i)}\bar{d}\bar{c})$, 
where $i\ne\pm j$,
\par\smallskip
(R6) \ $[T_{i,-i}(a),T_{-i,j}(d)]=
T_{ij}(ac)T_{-j,j}(-\lambda^{(\ep(j)-\ep(i))/2}\bar c a c)$,
where $i\ne\pm j$.
\par\smallskip
Relation (R1) coordinates two natural parametrisations of the same short
root subgroup $X_{ij}=X_{-j,-i}$. Relation (R2) expresses additivity of
the natural parametrisations. All other relations are various instances
of the Chevalley commutator formula. Namely, (R3) corresponds to the
case, where the sum of two roots is not a root, whereas (R4), and (R5)
correspond to the case of two short roots, whose sum is a short root,
and a long root, respectively. Finally, (R6) is the Chevalley commutator
formula for the case of a long root and a short root, whose sum is a root.
Observe that any two long roots are either opposite, or orthogonal, so
that their sum is never a root.


\section{Relative subgroups}\label{sec4}

In this section we recall definitions and basic facts concerning relative
subgroups. For the proofs of these results, see

\subsection{Relative subgroups}\Label{relative} One associates with a form ideal $(I,\Gamma)$
the following four relative subgroups.
\par\smallskip
$\bullet$ The subgroup $\FU(2n,I,\Gamma)$ generated by elementary unitary
transvections of level $(I,\Gamma)$,
$$ \FU(2n,I,\Ga)=\big\langle T_{ij}(a)\mid \
a\in I\text{ if }i\neq\pm j\text{ and }
a\in\lambda^{-(\epsilon(i)+1)/2}\Gamma\text{ if }i=-j\big\rangle. $$
\par\smallskip
$\bullet$ The {\it relative elementary subgroup\/} $\EU(2n,I,\Gamma)$
of level $(I,\Gamma)$, defined as the normal closure of $\FU(2n,I,\Gamma)$
in $\EU(2n,\FormR)$,
$$ \EU(2n,I,\Ga)={\FU(2n,I,\Ga)}^{\EU(2n,\FormR)}. $$
\par\smallskip
$\bullet$ The {\it principal congruence subgroup\/} $\GU(2n,I,\Ga)$ of level
$(I,\Ga)$ in $\GU(2n,A,\Lambda)$ consists of those $g\in \GU(2n,A,\Lambda)$,
which are congruent to $e$ modulo $I$ and preserve $f(u,u)$ modulo $\Ga$,
$$ f(gu,gu)\in f(u,u)+\Ga, \qquad u\in V. $$
\par\smallskip
$\bullet$ The full congruence subgroup $\CU(2n,I,\Gamma)$ of level
$(I,\Gamma)$, defined as
$$ \CU(2n,I,\Ga)=\big\{ g\in \GU(2n,A,\Lambda) \mid
[g,\GU(2n,A,\Lambda)]\subseteq \GU(2n,I,\Ga)\big\}. $$
\par\smallskip
In some books, including \cite{HO}, the group $\CU(2n,I,\Ga)$
is defined differently. However, in many important situations
these definitions yield the same group.


\subsection{Some basic lemmas}\Label{relativefacts}
Let us collect several basic facts, concerning relative groups,
which will be used in the sequel. The first one of them, see
\cite{BV3}, Lemma 5.2, asserts that
the relative elementary groups are $\EU(2n,A,\Lambda)$-perfect.

\begin{Lem}\label{hww3}
Suppose either $n\ge 3$ or $n=2$ and $I=\Lambda I+I\Lambda$.
Then
$$ \EU(2n,I,\Gamma)=[\EU(2n,I,\Gamma),\EU(2n,A,\Lambda)]. $$
\end{Lem}

The next lemma gives generators of the relative elementary subgroup
$\EU(2n,I,\Ga)$ as a subgroup. With this end, consider matrices
$$ Z_{ij}(a,c)={}^{T_{ji}(c)}T_{ij}(a)
=T_{ji}(c)T_{ij}(a)T_{ji}(-c), $$
\noindent
where $a\in I$, $c\in A$, if $i\neq\pm j$, and
$a\in\lambda^{-(\e(i)+1)/2}\Gamma$,
$c\in\lambda^{-(\e(j)+1)/2}\Lambda$, if $i=-j$.
The following result is \cite{BV3}, Proposition 5.1.
\begin{Lem}\label{genelm}
Suppose $n\ge 3$. Then
\begin{multline*}
\EU(2n,I,\Ga)=\big\langle Z_{ij}(a,c)\mid \
a\in I, c\in A\text{ if }i\neq\pm j\text{ and }\\
a\in\lambda^{-(\epsilon(i)+1)/2}\Gamma,
c\in\lambda^{-(\epsilon(j)+1)/2}\Lambda,
\text{ if }i=-j\big\rangle.
\end{multline*}
\end{Lem}
The following lemma was first established in~\cite{B1}, but remained
unpublished. See~\cite{HO} and~\cite{BV3}, Lemma 4.4, for published
proofs.
\begin{Lem}
The groups $\GU(2n,I,\Gamma)$ and $\CU(2n,I,\Gamma)$ are normal in
$\GU(2n,A,\Lambda)$.
\end{Lem}

In this form the following lemma was established in \cite{RNZ5},
Lemmas 7 and 8, see also \cite{RNZ4}, Lemma~1B for a
definitive exposition. Before that \cite{RNZ1}, Lemmas 21--23
only established weaker inclusions, with smaller left hand sides, 
or larger right hand sides.

\begin{Lem}
$(\FormR)$ be an associative form ring with $1$, $n\ge 3$, and 
let $(\FidealI)$ and $(\FidealJ)$ be two form ideals of $(\FormR)$. 
Then 
\begin{align*} \EU(2n,(\FidealI)\circ(\FidealJ))\le&\big[\FU(2n,\FidealI),\FU(2n,\FidealJ)\big]\le\\
&\big[\EU(2n,\FidealI),\EU(2n,\FidealJ)\big]\le\\
&\big[\GU(2n,\FidealI),\GU(2n,\FidealJ)\big] \le\GU(2n,(\FidealI)\circ(\FidealJ)).
\end{align*}
\end{Lem}



\section{Elementary commutators modulo $\EU(2n,(\FidealI)\circ(\FidealJ))$}

Now we embark on the proof of the second claim of Theorem~\ref{generators}.
Our first major goal is to prove that the commutator 
$[\FU(2n,\FidealI), \FU(2n,\FidealJ)]$ is central in $\EU(2n,\FormR)$,
modulo $\EU(2n,(\FidealI)\circ(\FidealJ))$. Namely, here we 
establish Theorem~\ref{equality}  and derive some corollaries thereof.
We prove the congruence in Theorem~\ref{equality}  separately for short
root positions, and then for long root positions.

\begin{Lem}\label{Op-1}
Let $(\FormR)$ be an associative form ring with $1$, $n\ge 3$, and let $(\FidealI)$, $(\FidealJ)$ be form ideals of $(\FormR)$.
For any $i\ne\pm j$ any $a\in I$, $b\in J$ and any
$x\in\EU(2n,\FormR)$, one has
$$ {}^x Y_{ij}(a,b)\equiv Y_{ij}(a,b) 
\pamod{\EU(2n,(\FidealI)\circ(\FidealJ))}. $$
\end{Lem}

\begin{proof}
Consider the elementary conjugate ${}^xY_{ij}(a,b)$. We argue by induction
on the length of $x\in\EU(2n,\FormR)$ in elementary generators. Let
$x=yT_{kl}(c)$, where $y\in \EU(2n,\FormR)$ is shorter than $x$.
\par
We start with the case $k\neq\pm l$.
\par\smallskip
$\bullet$ If $k,l\neq \pm i, \pm j$, then $T_{kl}(c)$ commutes with $z=Y_{ij}(a,b)$
and can be discarded.
\par\smallskip
$\bullet$ On the other hand, for any $h\neq \pm i, \pm j$  direct computations show that

\begin{align*}
&[T_{ih}(c),z]=T_{ih}(-abc-ababc)T_{jh}(-babc),\\\noalign{\vskip3truept}
&[T_{jh}(c),z]=T_{ih}(abac)T_{jh}(bac),\\
\noalign{\vskip3truept}
&[T_{hi}(c),z]=T_{ih}(cab)T_{jh}(-caba),\\
\noalign{\vskip3truept}
&[T_{hj}(c),z]=T_{ih}(cbab)T_{jh}(-cba-cbaba),
\end{align*}
Similarly, one has
\begin{alignat*}{1}
[T_{-i,h}(c),z]&=
[T_{-h,i}(-\lambda^{((\epsilon(h)+\epsilon(i))/2}c),z]r\\
&=T_{i,-h}(-\lambda^{((\epsilon(h)+\epsilon(i))/2}cab)T_{j,-h}(-\lambda^{((\epsilon(h)+\epsilon(i))/2}caba),\\
\noalign{\vskip3truept}
[T_{-j,h}(c),z]& =
[T_{-h,j}(-\lambda^{((\epsilon(h)+\epsilon(j))/2}c),z]\\
&=T_{i,-h}(-\lambda^{((\epsilon(h)+\epsilon(j))/2}cbab)
T_{j,-h}(-\lambda^{((\epsilon(h)+\epsilon(j))/2}cba-\lambda^{((\epsilon(h)+\epsilon(j))/2}cbaba),\\
\noalign{\vskip3truept}
[T_{h,-i}(c),z]& =
[T_{i,-h}(-\lambda^{(-(\epsilon(i)-\epsilon(h))/2}c),z] \\
                      &=T_{i,-h}(-\lambda^{(-(\epsilon(i)-\epsilon(h))/2}abac)T_{j,-h}(-\lambda^{(-(\epsilon(i)-\epsilon(h))/2}bac),\\
                      \noalign{\vskip3truept}
[T_{h,-j}(c),z]& =
[T_{j,-h}(-\lambda^{(-(\epsilon(j)-\epsilon(h))/2}c),z] \\
					  & =T_{i,-h}(-\lambda^{(-(\epsilon(j)-\epsilon(h))/2}abac)T_{j,-h}(-\lambda^{(-(\epsilon(j)-\epsilon(h))/2}bac)
\end{alignat*}
All factors on the right hand side belong already to 
$\EU(2n,(\FidealI)\circ(\FidealJ))$.
\par
If  $(k,l)=(\pm i, \pm j)$ or $(\pm j,\pm i)$, then we take an index  $h\ne \pm i, \pm j$ and rewrite $T_{kl}(c)$ as $[T_{k,h}(c),T_{h,l}(1)]$ and apply the previous items to get the same congruence modulo $\EU(2n,(\FidealI)\circ(\FidealJ))$.
\par
It remains to consider the case, where $k=-l$. 
\par\smallskip
$\bullet$ if $k\ne \pm i, \pm j$ then $T_{k,-k}(c)$ commutes with $z$ and can be discarded.
\par\smallskip
$\bullet$ Otherwise, we have
\begin{alignat*}{1}
[T_{i,-i}(c),z]=&T_{i,-i}(c-(1+ab+abab)c\overline{(1+ab+abab)})T_{j,-j}(-\lambda^{((\epsilon(i)-\epsilon(j))/2}babc\overline{bab})\\
&T_{i,-j}(\lambda^{((\epsilon(i)-\epsilon(j))/2}(1+ab+abab)c\overline{(bab)}),\\
\noalign{\vskip3truept}
[T_{j,-j}(c),z]=&T_{j,-j}(c-(1-ba)c\overline{(1-ba)})T_{i,-i}(\lambda^{((\epsilon(j)-\epsilon(i))/2}abac\overline{aba})\\& T_{i,-j}(-abac(1-\overline{ba})),\\
\noalign{\vskip3truept}
[T_{-i,i}(c),z]=&[T_{-i,i}(c),[T_{ij}(a),T_{ji}(b)]]\\=&[T_{-i,i}(c),[T_{-j,-i}(-\lambda^{((\epsilon(j)-\epsilon(i))/2}a),T_{-i,-j}(\lambda^{((\epsilon(i)-\epsilon(j))/2}b)]],\\
\noalign{\vskip3truept}
[T_{-j,j}(c),z]=&[T_{-j,j}(c),[T_{ij}(a),T_{ji}(b)]]\\ =&[T_{-j,-j}(c),[T_{-j,-i}(-\lambda^{((\epsilon(j)-\epsilon(i))/2}a),T_{-i,-j}(\lambda^{((\epsilon(i)-\epsilon(j))/2}b)]].
\end{alignat*}
The two last cases reduce to the first two. Hence 
all factors on the right belong to $\EU(2n,(\FidealI)\circ(\FidealJ))$.

We have shown that for $i\ne \pm j$,
$$ {}^xz\equiv {}^yz \pamod{\EU(2n,(\FidealI)\circ(\FidealJ))}. $$
\end{proof}

\begin{Lem}\label{Op-2}
Let $(\FormR)$ be an associative form ring with $1$, $n\ge 3$, and let $(\FidealI)$, $(\FidealJ)$ be form ideals of $(\FormR)$.
For any  $a\in \lambda^{-(\epsilon(i)+1)/2}\Gamma$, 
$b\in\lambda^{(\epsilon(i)-1)/2}\Delta$ and any 
$x\in\EU(2n,\FormR)$, 
one has
$$ {}^x Y_{i,-i}(a,b)\equiv Y_{i,-i}(a,b) \pamod{\EU(2n,(\FidealI)\circ(\FidealJ))}. $$
\end{Lem}

\begin{proof}
Denote  $Y_{i,-i}(a,b)=[T_{i,-i}(a),T_{-i,i}(b)]$ by $z$.  
\par\smallskip
$\bullet$ If  $(k,l)=(-i,i)$, then 
$$ [T_{-i,i}(c),z]=[T_{-i,i}(c), [T_{i,-i}(a),T_{-i,i}(b)]]\\
=[T_{-i,i}(c), Z_{-i,i}(b,a)]. $$
\noindent
The same computation as in Case 2 in Lemma~\ref{Op-1} shows that 
$$ [T_{-i,i}(c),z] \in \EU(2n,(\FidealI)\circ(\FidealJ)). $$
\par\smallskip
$\bullet$ If $(k,l)=(i,-i)$, then
\begin{alignat*}{1}
[T_{i,-i}(c),z]=&[T_{i,-i}(c), [T_{i,-i}(a),T_{-i,i}(b)]]=\\
&[T_{i,-i}(c),  [T_{-i,i}(b),T_{i,-i}(a)]^{-1}]=\\
&[T_{-i,i}(b),T_{i,-i}(a)]^{-1}[[T_{-i,i}(b),T_{i,-i}(a)], T_{i,-i}(c)] [T_{-i,i}(b),T_{i,-i}(a)].
\end{alignat*}
Now the  inner factor $[[T_{-i,i}(b),T_{i,-i}(a)], T_{i,-i}(c)]$ falls into the previous case, hence belongs to  $\EU(2n,(\FidealI)\circ(\FidealJ))$.
But then the same applies also to its conjugate
$$[T_{-i,i}(b),T_{i,-i}(a)]^{-1}\cdot
\Big[[T_{-i,i}(b),T_{i,-i}(a)], T_{i,-i}(c)\Big]
\cdot [T_{-i,i}(b),T_{i,-i}(a)]. $$
\par\smallskip
$\bullet$ If $k=i$ and $j\ne\pm k$, then 
\begin{multline*}
[T_{i,j}(c),z]=[T_{i,j}(c), [T_{i,-i}(a),T_{-i,i}(b)]]=
T_{i,j}(-(ab+abab)c) T_{-i,j}(-babc)\cdot\\
T_{-j,j}(-\lambda^{((\epsilon(j)-\epsilon(i))/2}\overline{ c}bab\overline{c} -\lambda^{\epsilon(j)}(\overline{c}bababc+\overline{c}babababc)). 
\end{multline*}
\noindent
Since $a\in\lambda^{-(\epsilon(i)+1)/2}\Gamma$ and 
$b\in\lambda^{(\epsilon(i)-1)/2}\Delta$, it follows that 
the right side belongs to $\EU(2n,(\FidealI)\circ(\FidealJ))$.
\par\smallskip
$\bullet$  if $k=-i$ and $j\ne \pm k$, then 
\begin{alignat*}{1}
[T_{-i,j}(c),z]=&[T_{-i,j}(c), [T_{i,-i}(a),T_{-i,i}(b)]]\\
=&[T_{-i,i}(b),T_{i,-i}(a)][T_{-i,j}(c), [T_{-i,i}(b),T_{i,-i}(a)]]^{-1}[T_{-i,i}(b),T_{i,-i}(a)]^{-1}.
\end{alignat*}
By the previous case, 
$$ [T_{-i,j}(c), [T_{-i,i}(b),T_{i,-i}(a)]]\in \EU(2n,(\FidealI)\circ(\FidealJ)). $$
\noindent
As above, normality of $\EU(2n,(\FidealI)\circ(\FidealJ))$ then implies 
that the whole right side belongs to 
$\EU(2n,(\FidealI)\circ(\FidealJ))$. 
\par\smallskip
$\bullet$ Finally, the case $l=\pm i$ and $k\ne \pm i$ reduces 
to the case $k=\pm i$ via relation (R1).
\par\smallskip
We have shown that  
$$ {}^xz\equiv {}^yz \pamod{\EU(2n,(\FidealI)\circ(\FidealJ))}. $$
\noindent
By induction we get that 
$$ {}^xz\equiv z\pamod{\EU(2n,(\FidealI)\circ(\FidealJ))}. $$
\end{proof}


In particular, these results 
immediately imply the following additivity property of
the elementary commutators with respect to its arguments.

\begin{The}\label{symb}
Let $R$ be an associative ring with $1$, $n\ge 3$, and let $(\FidealI)$, $(\FidealJ)$ 
be form ideals of $R$. Then for any  $ i\neq j$, and any
$a,a_1,a_2\in(\FidealI)$, $b,b_1,b_2\in(\FidealJ)$ one has
\begin{align*}
&Y_{ij}(a_1+a_2,b)\equiv  Y_{ij}(a_1,b)\cdot Y_{ij}(a_1,b) 
\pamod{\EU(2n,(\FidealI)\circ(\FidealJ))},\\
\noalign{\vskip3truept}
&Y_{ij}(a,b_1+b_2)\equiv  Y_{ij}(a,b_1)\cdot Y_{ij}(a,b_2) 
\pamod{\EU(2n,(\FidealI)\circ(\FidealJ))},\\
\noalign{\vskip3truept}
&Y_{ij}(a,b)^{-1}\equiv  Y_{ij}(-a,b)\equiv Y_{ij}(a,-b) 
\pamod{\EU(2n,(\FidealI)\circ(\FidealJ))},\\
\noalign{\vskip3truept}
&Y_{ij}(ab_1,b_2)\equiv Y_{ij}(a_1,a_2b)\equiv e
\pamod{\EU(2n,(\FidealI)\circ(\FidealJ))}\\
\noalign{\vskip3truept}
&Y_{i,-i}(\overline{b_1}ab_1,b_2)\equiv Y_{i,-i}(a_1,\overline{a_2}ba_2)\equiv e
\pamod{\EU(2n,(\FidealI)\circ(\FidealJ))}
\end{align*}
\end{The}
\begin{proof}
The first item can be derived from Lemma~\ref{Op-2.1} for 
$i\neq\pm j$  and Lemma~\ref{Op-2.2} for $i=-j$ as follows. By
definition
$$ Y_{ij}(a_1+a_2,b)=[T_{ij}(a_1+a_2),T_{ji}(b)]=
[T_{ij}(a_1)T_{ij}(a_2),T_{ji}(b)], $$
\noindent
and it only remains to apply multiplicativity of commutators in 
the first factor, and then apply Lemma~\ref{Op-2.1} and Lemma~\ref{Op-2.2} respectively. The second item is similar, 
and the third item follows. The last two items are obvious from 
the definition.
\end{proof}


\section{Unrelativisation}

Here we establish the first claim of Theorem~\ref{generators}, and thus also 
Theorems~\ref{equality2}, \ref{unrelative}  and \ref{T:4}. It immediately follows from the next two
lemmas, the first of which addresses the case of short roots, 
while the second one the case of long roots.

Recall that for the easier case of the general linear group over 
{\it commutative\/} rings this result was first established in 
2018 in our paper \cite{NZ1}. Then it was generalised to
arbitrary associative rings in 2019, together with the second 
claim of Theorem~\ref{generators}, see \cite{NZ2}. The proof of the following 
results exploit the same ideas as the proof of \cite{NZ2},
Lemma 4, but are noticeably more demanding from a technical viewpoint.

The following two lemmas address the case of short roots, where 
$i\ne \pm j$, and the case of long roots, where $i=-j$, respectively

\begin{Lem}\label{form-1}
Let $(\FormR)$ be an associative form ring with $1$, $n\ge 3$, and let $(\FidealI)$, $(\FidealJ)$ be form ideals of $(\FormR)$.
Suppose that  $a\in I$, $b\in J$, $r\in A$ and $i\ne \pm j$. Then
$$ [T_{ji}(a), Z_{ji}( b , r )]\in  [\FU(2n,\FidealI), \FU(2n, \FidealJ)]. $$
\end{Lem}

\begin{proof}
Without loss of generality, we may assume that $\ep(i)=\ep(j)$. Pick an $h \ne i, j$ with $\ep(h)=\ep(i)$. Then
$$ x=[T_{ji}(a), Z_{ji}( b , r )]=
T_{ji}(a)\cdot {}^{Z_{ji}( b , r )}T_{ji}(-a)
    =T_{ji}(a)\cdot {}^{Z_{ji}( b , r )}[T_{jh}(1),T_{hi}(-a)]. $$
  Thus,
  \begin{align*}
x={}&T_{ji}(a) [{}^{Z_{ji}( b , r )}T_{jh}(1),{}^{Z_{ji}( b , r )}T_{hi}(-a)]=\\
&T_{ji}(a) [T_{jh}(1- b  r )T_{ih}(- r  b  r ),T_{hj}(-a r  b  r )T_{hi}(-a(1- r  b ))]=\\
&T_{ji}(a) [T_{jh}(1)y,T_{hi}(-a)z],
  \end{align*}
   where
  \begin{multline*}
  y=T_{jh}(- b  r )T_{ih}(- r  b  r )\in \FU(2n,\FidealJ), \\
  z = T_{hj}(-a r  b  r )T_{hi}(a r  b )\in \FU(2n, (\FidealI)\circ(\FidealJ)).
  \end{multline*}
  Since $T_{hi}(-a)\in\FU(2n,\FidealI)$, the second factor of the above commutator belongs to $\FU(2n,\FidealI)$. Thus,
  \begin{eqnarray}\label{eqn:1}
[T_{jh}(1)y,T_{hi}(-a)z] &=&  {}^{T_{jh}(1)}[y,T_{hi}(-a)z]\cdot [{T_{jh}(1)},T_{hi}(-a)z].
  \end{eqnarray}
  Now, the first commutator on the right hand side 
  \begin{eqnarray*}
    {}^{T_{jh}(1)}[y,T_{hi}(-a)z] &=&  {}^{T_{jh}(1)}[T_{jh}(- b  r )T_{ih}(- r  b  r ),T_{hi}(-a)T_{hj}(-a r  b  r )T_{hi}(a r  b )].
  \end{eqnarray*}
  Expanding the commutator above by its second argument, we obtain
  \begin{eqnarray*}
    &&  {}^{T_{jh}(1)}[T_{jh}(- b  r )T_{ih}(- r  b  r ),T_{hi}(-a)T_{hj}(-a r  b  r )T_{hi}(a r  b )]\\
    &=&{}^{T_{jh}(1)} [T_{jh}(- b  r )T_{ih}(- r  b  r ),T_{hi}(-a)]\\
    &&\qquad \qquad\qquad {}^{T_{jh}(1)T_{hi}(-a)}[T_{jh}(- b  r )T_{ih}(- r  b  r ),T_{hj}(-a r  b  r )T_{hi}(a r  b )].
  \end{eqnarray*}
The second factor above belongs to $\EU(2n,(\FidealI)\circ(\FidealJ))$. And the first factor above equals
\begin{multline*}
{}^{T_{jh}(1)T_{jh}(- b  r )} [T_{ih}(- r  b  r ),T_{hi}(-a)]\cdot [T_{jh}(- b  r ),T_{hi}(-a)]\\={}^{T_{jh}(1)T_{jh}(- b  r )} [T_{ih}(- r  b  r ),T_{hi}(-a)]\cdot T_{ji}( b  r a)\\\in {}^{T_{jh}(1)T_{jh}(- b  r )} [T_{ih}(- r  b  r ),T_{hi}(-a)]\cdot\EU(2n,(\FidealI)\circ(\FidealJ)).
\end{multline*}
On the other hand, the second commutator of (\ref{eqn:1}) equals
  $$
[{T_{jh}(1)},T_{hi}(-a)]\cdot {}^{T_{hi}(-a)}[{T_{jh}(1)},z].
$$
The second commutator in the last expression belongs to $\EU(2n,(\FidealI)\circ(\FidealJ))$, and remains there after elementary conjugations, while the first commutator equals $T_{ij}(-a)$.

Summarising the above, we see that
$$
  x\in {}^{T_{ji}(a)T_{jh}(1)T_{jh}(- b  r )} [T_{ih}(- r  b  r ),T_{hi}(-a)]\cdot 
  \EU(2n,(\FidealI)\circ(\FidealJ))
$$
which belongs to $[\FU(2n,\FidealI), \FU(2n, \FidealJ)]$ by Lemma~\ref{Op-1}. 
\end{proof}


\begin{Lem}\label{form-2}
Let $(\FormR)$ be an associative form ring with $1$, $n\ge 3$, and let $(\FidealI)$, $(\FidealJ)$ be form ideals of $(\FormR)$.
Suppose that  $a\in\Gamma$, $b\in\Delta$ and $r\in\Lambda$. Then
$$[T_{-i,i}(a),Z_{-i,i}(b,r)]\in [\FU(2n,\FidealI),\FU(2n, \FidealJ)]. $$
\end{Lem}

\begin{proof} 
Without loss of generality, we may assume that $i>0$. Pick an $h>0$ with $h\ne i$. Then
\begin{multline*}
x=[T_{-i,i}(a), Z_{-i,i}( b , r )]=
T_{-i,i}(a)\cdot {}^{Z_{-i,i}( b , r )}T_{-i,i}(-a)=\\
T_{-i,i}(a)\cdot {}^{Z_{-i,i}( b , r )}\Big(T_{hi}(-a)\cdot
[T_{h,-h}(a),T_{-h,i}(1)]\Big). 
    \end{multline*}
\noindent
  Thus,
  \begin{align*}
x={}&T_{-i,i}( a )\cdot \Big({}{}^{Z_{-i,i}( b , r )}
T_{hi}(- a )\cdot 
[T_{h,-h}( a ),{}{}^{Z_{-i,i}( b , r )}T_{-h,i}(1)]\Big)=\\
&T_{-i,i}( a )\cdot T_{h,i}(- a (1- b  r ))\cdot T_{i,-h}(\lambda r  b  r \bar a )\cdot\Big[T_{h,-h}( a ),T_{-h,i}(1- r  b )\cdot T_{i,h}(\lambda r  b  r )\Big]
 \end{align*}
 \noindent
 Using additivity of root unipotents, we can rewrite this as
 $$ x=T_{-i,i}( a )T_{h,i}(- a )\cdot T_{h,i}(- a  b  r )
         T_{i,-h}(\lambda r  b  r \bar a )\cdot\Big[T_{h,-h}( a ),T_{-h,i}(1)T_{-h,i}(- r  b )\cdot T_{i,h}(\lambda r  b  r )\Big]. $$ 
\noindent
Clearly,
$$ T_{h,i}(- a  b  r )T_{i,-h}(\lambda r  b  r \bar a ) \in 
\EU(2n,(\FidealI)\circ(\FidealJ)). $$
\noindent
On the other hand, the commutator in the last expression equals
  \begin{multline*}
\Big[T_{h,-h}( a ),T_{-h,i}(1)T_{-h,i}(- r  b )\cdot 
T_{i,h}(\lambda r  b  r )\Big]=\\
\Big[T_{h,-h}( a ),T_{-h,i}(1)\Big]\cdot 
{}^{T_{-h,i}(1)}\Big[T_{h,-h}( a ), T_{-h,i}(- r  b )\cdot T_{i,h}(\lambda r  b  r )\Big]=\\
T_{h,i}( a  )T_{-i,i}(- a )\cdot {}^{T_{-h,i}(1)}
\Big[T_{h,-h}( a ), T_{-h,i}(- r  b )\cdot T_{i,h}(\lambda r  b  r )\Big].
\end{multline*}
\noindent
Again, clearly   
$$ \Big[T_{h,-h}( a ), T_{-h,i}(- r  b )\cdot T_{i,h}(\lambda r  b  r )\Big]\in [\FU(2n,\FidealI), \FU(2n, \FidealJ)]. $$
\noindent
On the other hand, the previous factors assemble to a left 
$T_{-i,i}(a)T_{h,i}(-a)$ conjugate 
of an element of $\EU(2n,(\FidealI)\circ(\FidealJ))$ ,
which is contained in $[\FU(2n,\FidealI), \FU(2n, \FidealJ)]$.
This proves Lemma~\ref{form-2}.
\end{proof}

Combined, these results imply the first claim of Theorem~\ref{generators}.

\section{Rolling over elementary commutators}

Now we pass to the final, and most difficult part of the
proof of Theorem~\ref{generators}, rolling an elementary commutator over 
to a different position. Since we assume $n\ge 3$, the case of
{\it short} root type elementary commutators is easy. It is
settled by
essentially the same calculation as for the general linear group
$\GL(n,R)$, $n\ge 3$, see \cite{NZ2,NZ3}. But for the case
of {\it long\/} root type elementary commutators we have to
imitate the proof of \cite{NZ4}, Theorems 4 and 5, for $\Sp(4,R)$.
In the presence of non-trivial involution, non-commutativity
and non-trivial form parameters this is quite a challenge.
In \S~12 we make some observations, 
to put this calculation in historical context.

\begin{Lem}\label{Op-2.1}
Let $(\FormR)$ be an associative form ring with $1$, $n\ge 3$, 
and let $(\FidealI)$, $(\FidealJ)$ be form ideals of $(\FormR)$. 
Then for any $i\neq \pm j$, any $h\ne\pm l$, and any 
$a\in I$, $b\in J$, $c_1,c_2\in A$, one has
$$ Y_{ij}(c_1ac_2,b)\equiv Y_{hl}(a,c_2bc_1)
\pamod{\EU(2n,(\FidealI)\circ(\FidealJ))}.  $$ 
\end{Lem}

\begin{proof}
Take any $h\neq\pm i,\pm j$, and rewrite the elementary 
commutator $z=Y_{ij}(c_1ac_2,b)$ on the left hand side of the
above congruence as follows
\begin{multline*}
z=\big[T_{ij}(c_1ac_2),T_{ji}(b)\big]
=T_{ij}(c_1ac_2)\cdot {}^{T_{ji}(b)}T_{ij}(-c_1ac_2)=\\
T_{ij}(c_1ac_2)\cdot {}^{T_{ji}(b)} [T_{hj}(ac_2),T_{ih}(c_1)]. 
\end{multline*}
\noindent
    Expanding the conjugation by $T_{ji}(b)$, we see that
\begin{multline*}
z=T_{ij}(c_1ac_2)\cdot  [{}^{T_{ji}(b)}T_{hj}(ac_2),
{}^{T_{ji}(b)}T_{ih}(c_1)]=\\
T_{ij}(c_1ac_2)\cdot  \Big[[T_{ji}(b),T_{hj}(ac_2)]
T_{hj}(ac_2),T_{ih}(c_1)[T_{ih}(-c_1),{T_{ji}(b)}]\Big]=\\
        T_{ij}(c_1ac_2)\cdot  \Big[T_{hi}(-ac_2b)T_{hj}(ac_2),T_{ih}(c_1)T_{jh}(bc_1)\Big].
\end{multline*}
\noindent
Now, the first factor $T_{hi}(-ac_2b)$ of the first argument in 
this last commutator already belongs to the group 
$\FU(2n,(\FidealI)\circ(\FidealJ))$.
Thus, as above,
$$ z\equiv  T_{ij}(c_1ac_2)\cdot  \Big[T_{hj}(ac_2),T_{ih}(c_1)T_{jh}(bc_1)\Big] \pamod{\EU(2n,(\FidealI)\circ(\FidealJ))}. $$
\noindent
Using multiplicativity of the commutator w.r.t. the second argument, cancelling the first two factors of the resulting expression, and then applying Lemma~\ref{Op-1} we see that
$$ z\equiv
{}^{T_{ih}(c_1)}\big[T_{hj}(ac_2),T_{jh}(bc_1)\big]
\equiv \big[T_{hj}(ac_2),T_{jh}(bc_1)\big]
\pamod{\EU(2n,(\FidealI)\circ(\FidealJ))}. $$
\par
On the other hand, choosing another index $l\neq\pm j,\pm h$ and rewriting the commutator
$\big[T_{hj}(ac_2),T_{jh}(bc_1)\big]$ on the right hand side of the
last congruence as
$$ \big[T_{hj}(ac_2),T_{jh}(bc_1)\big]=
\big[[T_{hl}(a),T_{lj}(c_2)],T_{jh}(bc_1)\big], $$
\noindent
by the same argument we get the congruence
$$ z\equiv\big[T_{hj}(ac_2),T_{jh}(bc_1)\big]\equiv
\big[T_{hl}(a),T_{lh}(c_2bc_1)\big]
\pamod{\EU(2n,(\FidealI)\circ(\FidealJ))}. $$
\par 
Obviously, for $n\ge 3$ we can pass from any position $(i,j)$,
$i\neq j$, to any other such position $(k,m)$, $k\neq\pm m$, 
by a sequence of at most three such elementary moves.
\end{proof}


\begin{Lem}\label{Op-2.2}
Let $(\FormR)$ be an associative form ring with $1$, $n\ge 3$, 
and let $(\FidealI)$, $(\FidealJ)$ be form ideals of $(\FormR)$. 
Then for any $ -n\le i\le n$, any $-n\le k\le n$, and any 
$a\in \lambda^{-(\epsilon(i)+1)/2}\Gamma$, 
$b\in \lambda^{(\epsilon(i)-1)/2}\Delta$, $c\in A$, one has
$$ Y_{i,-i}(ca\bar c,b)\equiv Y_{k,-k}(\lambda^{(\epsilon(i)-\epsilon(k))/2}a,-\lambda^{(\epsilon(k)-\epsilon(i))/2}\bar c  bc)\pamod{\EU(2n,(\FidealI)\circ(\FidealJ))}.  $$ 
\end{Lem}

\begin{proof}
Rewrite the elementary commutator $z=Y_{i,-i}(ca\bar c,b)$ 
on the left hand side of the above congruence as follows
\begin{multline*}
z=T_{i,-i}(ca\bar c )\cdot 	^{T_{-i,i}(b)}T_{i,-i}(-ca\bar c)=\\
T_{i,-i}(ca\bar c)\cdot ^{T_{-i,i}(b)}\Big(T_{i,-k}(\lambda^{(\epsilon(i)-\epsilon(k))/2}ca)[T_{i,k}( c),T_{k,-k}(-\lambda^{(\epsilon(i)-\epsilon(k))/2}a)]\Big).
\end{multline*}
\noindent
Expanding the conjugation by $T_{-i,i}(b)$, we see that
$$ z=T_{i,-i}(ca\bar c)\cdot ^{T_{-i,i}(b)}T_{i,-k}(\lambda^{(\epsilon(i)-\epsilon(k))/2}ca)\cdot\Big[{}^{T_{-i,i}(b)}T_{i,k}(c),{}^{T_{-i,i}(b)}T_{k,-k}(-\lambda^{(\epsilon(i)-\epsilon(k))/2}a)\Big]. $$
\noindent
Clearly, the first two factors 
$$ y=T_{i,-i}(ca\bar c)\cdot^{T_{-i,i}(b)}T_{i,-k}(\lambda^{(\epsilon(i)-\epsilon(k))/2}ca) $$
\noindent
can be rewritten as
$$ y=T_{i,-i}(ca\bar c)\cdot
\big[{T_{-i,i}(b)},
T_{i,-k}(\lambda^{(\epsilon(i)-\epsilon(k))/2}ca)\big]\cdot
T_{i,-k}(\lambda^{(\epsilon(i)-\epsilon(k))/2}ca) $$
\noindent
which gives us the following congruence
$$ y
\equiv T_{i,-i}(ca\bar c)T_{i,-k}(\lambda^{(\epsilon(i)-\epsilon(k))/2}ca) \pamod{\EU(2n,(\FidealI)\circ(\FidealJ))}. $$
\par
On the other hand, the commutator 
$$ u=\Big[{}^{T_{-i,i}(b)}T_{i,k}(c),T_{k,-k}(-\lambda^{(\epsilon(i)-\epsilon(k))/2}a)\Big] $$ 
\noindent
in the expression of $z$ equals
$$ u=\Big[T_{-i,k}(bc)
T_{-k,k}(-\lambda^{(\epsilon(k)-\epsilon(i))/2}\bar c b c)
T_{i,k}(c), T_{k,-k}(-\lambda^{(\epsilon(i)-\epsilon(k))/2}a)\Big]. $$
\noindent
Expanding this last expression, we get
\begin{multline*}
 u={}^{x}[T_{i,k}(c), 
T_{k,-k}(-\lambda^{(\epsilon(i)-\epsilon(k))/2}a)]\cdot\\
{}^{y}[T_{-k,k}(-\lambda^{(\epsilon(k)-\epsilon(i))/2}\bar c bc),
T_{k,-k}(-\lambda^{(\epsilon(i)-\epsilon(k))/2}a)]\cdot\\
[T_{-i,k}(bc),T_{k,-k}(-\lambda^{(\epsilon(i)-\epsilon(k))/2}a)], 
\end{multline*}
\noindent
where 
$$ x=T_{-i,k}(bc)
T_{-k,k}(-\lambda^{(\epsilon(k)-\epsilon(i))/2}\bar c b c),\qquad y=T_{-i,k}(bc). $$
\noindent
It is easy to see that 
$$ [T_{-i,k}(bc),T_{k,-k}(-\lambda^{(\epsilon(i)-\epsilon(k))/2}a)] 
\in\EU(2n,(\FidealI)\circ(\FidealJ)), $$
\noindent
so we can drop it. Further, by Lemma~\ref{Op-2},
modulo $\EU(2n,(\FidealI)\circ(\FidealJ))$ the second
factor can be simplified as follows
\begin{multline*}
{}^{y}[T_{-k,k}(-\lambda^{(\epsilon(k)-\epsilon(i))/2}\bar c b c),
T_{k,-k}(-\lambda^{(\epsilon(i)-\epsilon(k))/2}a)]\equiv\\
[T_{-k,k}(-\lambda^{(\epsilon(k)-\epsilon(i))/2}\bar c bc),
T_{k,-k}(-\lambda^{(\epsilon(i)-\epsilon(k))/2}a)] 
\pamod{\EU(2n,(\FidealI)\circ(\FidealJ))}
\end{multline*}
\noindent
But by Theorem~\ref{symb} one has
\begin{multline*}
[T_{-k,k}(-\lambda^{(\epsilon(k)-\epsilon(i))/2}\bar c bc),T_{k,-k}(-\lambda^{(\epsilon(i)-\epsilon(k))/2}a)]\equiv\\ 
[T_{k,-k}(\lambda^{(\epsilon(i)-\epsilon(k))/2}a), T_{-k,k}(-\lambda^{(\epsilon(k)-\epsilon(i))/2}\bar c bc)] 
\pamod{\EU(2n,(\FidealI)\circ(\FidealJ))}.
\end{multline*}

\par
Summarising the above, we get
\begin{multline*}
z\equiv T_{i,-i}(a)
T_{i,-k}(\lambda^{(\epsilon(k)-\epsilon(i))/2}ca)\cdot 
{}^{x}[T_{i,k}(c), 
T_{k,-k}(-\lambda^{(\epsilon(k)-\epsilon(i))/2}a)]\cdot\\
[T_{k,-k}(\lambda^{(\epsilon(i)-\epsilon(k))/2}a), T_{-k,k}(-\lambda^{(\epsilon(k)-\epsilon(i))/2}\bar c bc)] 
\pamod{\EU(2n,(\FidealI)\circ(\FidealJ))}.
 \end{multline*}
 \noindent
Thus, to finish the proof it suffices to show that  
$$ v=T_{i,-i}(a)T_{i,-k}(\lambda^{(\epsilon(k)-\epsilon(i))/2}ca)\cdot ^{x}[T_{i,k}(c), T_{k,-k}(-\lambda^{(\epsilon(k)-\epsilon(i))/2}a)]
$$
\noindent
belongs to $\EU(2n,(\FidealI)\circ(\FidealJ))$. Clearly,
$$ v=T_{i,-i}(ca\bar c)
T_{i,-k}(\lambda^{(\epsilon(k)-\epsilon(i))/2}ca)\cdot 
{}^{x} T_{i,-k}(-\lambda^{(\epsilon(k)-\epsilon(i))/2}ca)
T_{i,-i}(-ca\bar c), $$
\noindent
can be rewritten as
\begin{multline*}
v=[T_{i,-i}(ca\bar c)
T_{i,-k}(\lambda^{(\epsilon(k)-\epsilon(i))/2}ca),x]=\\
[T_{i,-i}(ca\bar c)T_{i,-k}(\lambda^{(\epsilon(k)-\epsilon(i))/2}ca), 
T_{-i,k}(bc)T_{-k,k}(-\lambda^{(\epsilon(k)-\epsilon(i))/2}\bar cbc)].
\end{multline*}
\noindent
Expanding this last commutator w.r.t. its first and second arguments,
we express it as the product of elementary conjugates 
of the four following commutators
\par\smallskip
$\bullet$ $[T_{i,-i}(ca\bar c), T_{-i,k}(bc)]$,
\par\smallskip
$\bullet$ $[T_{i,-i}(ca\bar c), T_{-k,k}(-\lambda^{(\epsilon(k)-\epsilon(i))/2}\bar c bc)]$,
\par\smallskip
$\bullet$ $[T_{i,-k}(\lambda^{(\epsilon(k)-\epsilon(i))/2}ca), T_{-i,k}(bc)]$, 
\par\smallskip
$\bullet$ 
$[T_{i,-k}(\lambda^{(\epsilon(k)-\epsilon(i))/2}ca), T_{-k,k}(-\lambda^{(\epsilon(k)-\epsilon(i))/2}\bar c bc)]$.
\par\smallskip\noindent
A direct computation convinces us that each of these commutators belongs to the elementary subgroup 
$\EU(2n,(\FidealI)\circ(\FidealJ))$. This finishes the 
proof of lemma, and thus also of Theorem~\ref{generators}. 
\end{proof}


\section{Mat[ch]ing elementary commutators of different root lengths}

In this section we prove a congruence connecting elementary
commutators of long root type with those of short root type.
In the case, where one of the relative form parameters is
as small as possible (=minimal), this congruence can be used 
to eliminate long root type elementary commutators. On the
other hand when one of the relative form parameters is as large as possible (=equals the corresponding ideal), one can abandon
short root type elementary commutators.

\begin{Lem}\label{long-short}
  Let $(\FormR)$ be an associative form ring with $1$, $n\ge 3$, and let $(\FidealI)$ , $(\FidealJ)$
  be form ideals of $(\FormR)$. Then for any  $ -n\le i\le n$, any
  $-n\le k\le n$, and $a\in I$, $b\in \lambda^{(\epsilon(i)-1)/2} \Delta$,  one has
$$\Big[T_{i,-i}\big(a-\lambda^{\ep(-i)} \bar a \big),T_{-i,i}\big(b\big)\Big] \equiv \big[ T_{i,k}(a), T_{k,i}(b)\big]
\pamod{\EU(2n,(\FidealI)\circ(\FidealJ))}. $$
\end{Lem}

\begin{proof}
Pick an index $k\ne \pm i$, and rewrite the elementary commutator 
$z=\Big[T_{i,-i}\big(a-\lambda^{\ep(-i)} \bar a \big),T_{-i,i}\big(b\big)\Big]$ on the left hand side as
$$ z=\Big[[T_{k,-i}(-1),T_{i,k}(a)], T_{-i,i}\big(b\big)\Big]=
\Big[ {}^{T_{k,-i}(-1)}T_{i,k}(a)\cdot T_{i,k}(-a), 
T_{-i,i}\big(b\big)\Big]. $$
\noindent
Using multiplicativity of the commutator w.r.t the first argument, 
we see
$$
z={}^{T_{k,-i}(-1)T_{i,k}(a)T_{k,-i}(1)}[T_{i,k}(-a),T_{-i,i}\big(b\big)]\cdot \Big[ {}^{T_{k,-i}(-1)}T_{i,k}(a), T_{-i,i}\big(b\big)\Big]. $$
\noindent
The first factor belongs to $\EU(2n,(\FidealI)\circ(\FidealJ))$, so we leave it out. Thus, $z$ is congruent modulo this subgroup to
\begin{multline*}
\Big[ {}^{T_{k,-i}(-1)}T_{i,k}(a), T_{-i,i}\big(b\big)\Big]= 
{}^{T_{k,-i}(-1)}\Big[ T_{i,k}(a), {}^{T_{k,-i}(1)}T_{-i,i}\big(b\big)\Big]=\\
={}^{T_{k,-i}(-1)}\Big[ T_{i,k}(a), [{T_{k,-i}(1)},T_{-i,i}\big(b\big)]T_{-i,i}\big(b\big)\Big]=\\
{}^{T_{k,-i}(-1)}\Big[ T_{i,k}(a), T_{k,i}\big(b\big)
T_{k,-k}\big(\lambda^{(\ep(-i)-\ep(k))/2}(b)\big)T_{-i,i}\big(b\big)\Big].
\end{multline*} 
\noindent
Expanding this last commutator w.r.t the second argument, we see
that the second and the third factors belong to $\EU(2n,(\FidealI)\circ(\FidealJ))$, so that we can leave them out. Now we have
$$ z\equiv{}^{T_{k,-i}(-1)}\Big[ T_{i,k}(a), T_{k,i}\big(b)\Big]
\pamod{\EU(2n,(\FidealI)\circ(\FidealJ))}, $$
\noindent
as claimed. 
\end{proof}

\begin{Cor}
In conditions of Lemma~$\ref{long-short}$ further assume that $b=b'-\lambda^{\ep(i)}\overline{b'}$ for some $b'\in J$, then 
$$ \Big[T_{i,-i}\big(a-\lambda^{\ep(-i)} \bar a \big),T_{-i,i}\big(b-\lambda^{\ep(i)} \bar b\big)\Big] \equiv \big[ T_{i,k}(a), T_{k,i}(b')\big]\cdot \big[T_{i,k}(a),T_{k,i}(-\lambda^{\ep(i)} \overline{b'})\big]
$$
\noindent
modulo $\EU(2n,(\FidealI)\circ(\FidealJ))$.
\end{Cor}

\begin{proof}
Keep the notation from the proof of Lemma~\ref{long-short}.
Under this additional assumption one has
$$ z\equiv {}^{T_{k,-i}(-1)}\Big[ T_{i,k}(a), T_{k,i}(b')T_{k,i}(-\lambda^{\ep(i)} \overline{b'})\Big]\pamod{\EU(2n,(\FidealI)\circ(\FidealJ))}. $$
\noindent
Expanding the commutator w.r.t the second argument again, 
we see that 
\begin{multline*}
{}^{T_{k,-i}(-1)}\Big[ T_{i,k}(a), T_{k,i}(b')T_{k,i}(-\lambda^{\ep(i)} \overline{b'})\Big]=\\
{}^{T_{k,-i}(-1)}\Big(\big[ T_{i,k}(a), T_{k,i}(b')\big]\cdot {}^{ T_{k,i}(b')}\big[ T_{i,k}(a),T_{k,i}(-\lambda^{\ep(i)} \overline{b'})\big]\Big).
\end{multline*}
\noindent
Applying Lemma~\ref{Op-1}, we get 
$$
z\equiv \big[ T_{i,k}(a), T_{k,i}(b')\big]\cdot \big[T_{i,k}(a),T_{k,i}(-\lambda^{\ep(i)} \overline{b'})\big]
\pamod{\EU(2n,(\FidealI)\circ(\FidealJ))}, $$
\noindent
as claimed. 
\end{proof}

\begin{Cor}
If $I=\Gamma$ or $J=\Delta$ then for the second type of 
generators in Theorem~$\ref{generators}$ it suffices to take one pair $(h,-h)$.
\end{Cor}

\begin{Cor}
If $\Gamma=I\cap\Lambda_{\min}$ or $\Delta=J\cap \Lambda_{min}$ then 
for the second type of generators in Theorem~$\ref{generators}$ it suffices to 
take one pair $(h,k)$, $h\neq\pm k$.
\end{Cor}


\section{Triple and quadruple commutators}

Actually Theorem~\ref{T:7} easily follows by induction on $m$ from
the following two special cases, triple commutators, and quadruple commutators.

\begin{Lem}\label{triple}
Let $(\FormR$) be any associative form ring with $1$, let $n\ge 3$, 
and let 
$(\FidealI)$, $(\FidealJ)$, $(\FidealK)$,  be form ideals of $(\FormR)$.  Then
\begin{multline*}
\big[\big[\EU(2n,\FidealI),\EU(2n,\FidealJ)\big],\EU(2n,\FidealK)\big]=\\
\big[\EU(2n,(\FidealI)\circ (\FidealJ)),\EU(2n,\FidealK)\big]. 
\end{multline*}
\end{Lem}

\begin{proof}
First of all, observe that the generators of the first type in Theorem~\ref{generators} belong to
$\EU(2n,(\FidealI)\circ (\FidealJ))$. Thus, forming their commutators with $T_{h,k}(c)\in \EU(2n,\FidealK)$ will bring us inside 
$[\EU(2n,(\FidealI)\circ (\FidealJ)),\EU(2n,\FidealK)]$. 
\par
Next, let $Y_{i,j}(a,b)=[T_{i,j}(a),T_{j,i}(b)]$ a typical generator of the second type of the commutator subgroup $\big[\EU(2n,\FidealI),\EU(2n,\FidealJ)\big]$ with $T_{i,j}(a)\in \EU(2n,\FidealI)$ and $T_{j,i}(b)\in \EU(2n,\FidealJ)$.
\par
From Lemma~\ref{Op-1} and Lemma~\ref{Op-2} we know that ${}^xY_{i,j}(a,b)=Y_{i,j}(a,b)z$, for some 
$z\in\EU(2n,(\FidealI)\circ (\FidealJ))$, and thus for any $T_{h,k}(c)\in \EU(2n,\FidealK)$,
$$ \big[{}^xY_{i,j}(a,b),T_{k,l}(c)\big]=
\big[Y_{i,j}(a,b)z,T_{k,l}(c)\big]={}^{Y_{ij}(a,b)}[z,T_{k,l}(c)]\cdot
[Y_{i,j}(a,b),T_{k,l}(c)]. $$ 
\noindent
The first of these commutators also belongs to
$$ \big[\EU(2n,(\FidealI)\circ (\FidealJ)),\EU(2n,\FidealK)\big], $$
\noindent 
and stays there after elementary conjugations. Let us concentrate 
at the second one.
\par\smallskip\noindent
{\bf Case 1.} When $i\ne \pm j$ the same 
analysis as in the proof of Lemma~\ref{Op-1}, shows that:
\par\smallskip
$\bullet$ If $k\ne -l$ and  $k,l\neq\pm  i,\pm j$, then $T_{k,l}(c)$ commutes with 
$Y_{i,j}(a,b)$.
\par\smallskip
$\bullet$ For any $h\neq\pm  i,\pm j$ the formulas 
for $Y_{ij}(a,b)$ and $Y_{ij}(a,b)^{-1}$ given in the proof of 
Lemma~\ref{Op-1} immediately imply that
\begin{alignat*}{1}
[z,T_{ih}(c)]&=T_{jh}(babc)T_{ih}(abc+ababc),\\
\noalign{\vskip 3truept}
[z,T_{jh}(c)]&=T_{jh}(-bac)T_{ih}(-abac),\\
\noalign{\vskip 3truept}
[z,T_{hi}(c)]&=T_{jh}(caba)T_{ih}(-cab),\\
\noalign{\vskip 3truept}
[z,T_{hj}(c)]&=T_{jh}(cba+cbaba)T_{ih}(-cbab),
\end{alignat*}
\noindent
and similarly
\begin{alignat*}{1}
[z,T_{-i,h}(c)]&=
[z,T_{-h,i}(-\lambda^{((\epsilon(h)+\epsilon(i))/2}c)]=\\
&T_{j,-h}(\lambda^{((\epsilon(h)+\epsilon(i))/2}caba)T_{i,-h}(\lambda^{((\epsilon(h)+\epsilon(i))/2}cab),\\
\noalign{\vskip 3truept}
[z,T_{-j,h}(c)]& =[z,T_{-h,j}(-\lambda^{((\epsilon(h)+\epsilon(j))/2}c)]=\\
&T_{j,-h}(\lambda^{((\epsilon(h)+\epsilon(j))/2}cba+\lambda^{((\epsilon(h)+\epsilon(j))/2}cbaba) T_{i,-h}(\lambda^{((\epsilon(h)+\epsilon(j))/2}cbab),\\
\noalign{\vskip 3truept}
[z,T_{h,-i}(c)]& =
[z,T_{i,-h}(-\lambda^{(-(\epsilon(i)-\epsilon(h))/2}c)]= \\
&=T_{j,-h}(\lambda^{(-(\epsilon(i)-\epsilon(h))/2}bac)T_{i,-h}(\lambda^{(-(\epsilon(i)-\epsilon(h))/2}abac),\\
\noalign{\vskip 3truept}
[z,T_{h,-j}(c)]& =
[z,T_{j,-h}(-\lambda^{(-(\epsilon(j)-\epsilon(h))/2}c)] \\
&T_{j,-h}(-\lambda^{((\epsilon(j)-\epsilon(h))/2}bac)T_{i,-h}(\lambda^{(-(\epsilon(j)-\epsilon(h))/2}abac)
\end{alignat*}
\par\noindent
All factors on the right hand side belong already to 
$\EU\big(2n,((\FidealI)\circ (\FidealJ))\circ (\FidealK)\big)$.
\par
If  $(k,l)=(\pm i, \pm j)$ or $(\pm j,\pm i)$, then we take an index  $h\ne \pm i, \pm j$ and rewrite $T_{kl}(c)$ as $[T_{k,h}(c),T_{h,l}(1)]$ and apply the previous items to get it belongs to $\big[\EU(2n,(\FidealI)\circ (\FidealJ)),\EU(2n,\FidealK)\big]$.
\par\smallskip

On the other hand, for $k=-l$ we have:
\par\smallskip
$\bullet$ If $k\ne \pm i,\pm j$, then $T_{k,-k}(c)$ commutes 
with $z$ and can be discarded. 
\par\smallskip
$\bullet$ Otherwise, we have
\begin{align*}
&[z,T_{i,-i}(c)]=T_{i,-j}(-\lambda^{((\epsilon(i)-\epsilon(j))/2}(1+ab+abab)c\overline{(bab)})T_{j,-j}(\lambda^{((\epsilon(i)-\epsilon(j))/2}babc\overline{bab})\\
&\hskip 2.5truein T_{i,-i}(-c+(1+ab+abab)c\overline{(1+ab+abab)}),\\
\noalign{\vskip 3truept}
&[z,T_{j,-j}(c)]=T_{i,-j}(abac(1-\overline{ba}))T_{i,-i}(-\lambda^{((\epsilon(j)-\epsilon(i))/2}abac\overline{aba})\cdot\\
&\hskip 3.5truein T_{j,-j}(-c+(1-ba)c\overline{(1-ba)}),\\
\noalign{\vskip 3truept}
&[z,T_{-i,i}(c)]=[[T_{ij}(a),T_{ji}(b)],T_{-i,i}(c)]=\\
&\hskip 1.5truein [[T_{-j,-i}(-\lambda^{((\epsilon(j)-\epsilon(i))/2}a),T_{-i,-j}(\lambda^{((\epsilon(i)-\epsilon(j))/2}b)],T_{-i,i}(c)],\\
\noalign{\vskip 3truept}
&[z,T_{-j,j}(c)]=[[T_{ij}(a),T_{ji}(b)],T_{-j,j}(c)]=\\ 
&\hskip 1.5truein [[T_{-j,-i}(-\lambda^{((\epsilon(j)-\epsilon(i))/2}a),T_{-i,-j}(\lambda^{((\epsilon(i)-\epsilon(j))/2}b)],T_{-j,-j}(c)].
\end{align*}
The two last cases reduce to the first two. In each case 
the resulting expressions belong to 
$\EU\big(2n,((\FidealI)\circ (\FidealJ))\circ (\FidealK)\big)$.
  

\par\smallskip\noindent
{\bf Case 2.} When $i=-j$ the same analysis as in the proof of Lemma~\ref{Op-2}, shows that:
\par\smallskip
$\bullet$ If  $(k,l)=(-i,i)$, then 
$$ [z,T_{-i,i}(c)]=[ [T_{i,-i}(a),T_{-i,i}(b)],T_{-i,i}(c)]
=[Z_{-i,i}(b,a),T_{-i,i}(c)]. $$
\noindent
Now, the same computation as  in Lemma~\ref{form-2} 
shows that 
$$ [z,T_{-i,i}(c)] \in \EU(2n,(\FidealI)\circ(\FidealJ)). $$
\par\smallskip
$\bullet$ If  $(k,l)=(i,-i)$, then
\begin{multline*}
[z,T_{i,-i}(c)]=[ [T_{i,-i}(a),T_{-i,i}(b)],T_{i,-i}(c)]=
[  [T_{-i,i}(b),T_{i,-i}(a)]^{-1},T_{i,-i}(c)]\\
=[T_{-i,i}(b),T_{i,-i}(a)]\cdot
[T_{i,-i}(c),[T_{-i,i}(b),T_{i,-i}(a)] ]\cdot 
[T_{-i,i}(b),T_{i,-i}(a)]^{-1}.
\end{multline*}
\noindent
By the previous subcase,
$$ [T_{i,-i}(c),[T_{-i,i}(b),T_{i,-i}(a)] ]\in \EU(2n,(\FidealI)\circ(\FidealJ)). $$ 
\noindent
But then its conjugates also stay therein.

  
\par\smallskip
$\bullet$ If $k=i$ and $j\ne \pm k$, then 
\begin{multline*}
 [z,T_{i,j}(c)]=[ [T_{i,-i}(a),T_{-i,i}(b)],T_{i,j}(c)]=\\ 
T_{-j,j}(\lambda^{((\epsilon(j)-\epsilon(i))/2}\overline{ c}bab\overline{c} \lambda^{\epsilon(j)}(\overline{c}bababc+\overline{c}babababc))\cdot T_{-i,j}(babc)T_{i,j}((ab+abab)c) 
\end{multline*}
\noindent
Since $a\in \lambda^{-(\epsilon(i)+1)/2}\Gamma$ and $b\in\lambda^{(\epsilon(i)-1)/2}\Delta$, it follows that the right 
hand side belongs to $\EU(2n,(\FidealI)\circ(\FidealJ))$.
\par\smallskip
$\bullet$ If $k=-i$ and $j\ne \pm k$, then 
\begin{multline*}
[z,T_{-i,j}(c)]=[ [T_{i,-i}(a),T_{-i,i}(b)],T_{-i,j}(c)]=\\
[T_{-i,i}(b),T_{i,-i}(a)]^{-1}\cdot 
[T_{-i,j}(c), [T_{-i,i}(b),T_{i,-i}(a)]]\cdot [T_{-i,i}(b),T_{i,-i}(a)].
\end{multline*}
\noindent
By the previous subcase, 
$$ [T_{-i,j}(c), [T_{-i,i}(b),T_{i,-i}(a)]]\in \EU(2n,(\FidealI)\circ(\FidealJ)). $$
\noindent
But then its conjugates also stay therein.
\par\smallskip
$\bullet$ Finally, using relation $(R1)$ the subcase $l=\pm i$ 
and $k\ne\pm i$ is readily reduced to the subcases, where 
$k=\pm i$.
\end{proof}

Now, for $n\ge 4$ the only new case of quadruple commutators
is considered in the following lemma, which immediately follows
from Lemma~\ref{triple} and Theorem~\ref{equality}. Of course, for the outstanding
case $n=3$ it requires a separate proof. All our assaults on 
this remaining case were crippled by forbidding calculations.

\begin{Lem}\label{quadruple}
Let $(\FormR$) be any associative form ring with $1$ and let 
$(\FidealI)$, $(\FidealJ)$, $(\FidealK)$, $(\FidealL)$  be form ideals of $(\FormR)$. If either $n\ge 4$ or there exists an ideal equals its corresponding relative form parameter and $n\ge3$, then
\begin{multline*}
 \Big[\big[\EU(2n,\FidealI),\EU(2n,\FidealJ)\big],\big[\EU(2n,\FidealK),\EU(2n,\FidealL)\big]\Big]=\\
\big[\EU(2n,(\FidealI)\circ (\FidealJ)),\EU(2n,(\FidealK)\circ (\FidealL))\big]. 
\end{multline*}
\end{Lem}

\begin{proof}
From the previous lemma we already know that 
\begin{multline*}
\Big[\EU(2n,(\FidealI)\circ(\FidealJ)),\big[\EU(2n,\FidealK),\EU(2n,\FidealL)\big]\Big]=\\
\Big[\EU(2n,(\FidealI)\circ (\FidealJ)),\EU(2n,(\FidealK)\circ (\FidealL))\Big]
\end{multline*}
and that
\begin{multline*}
\Big[\big[\EU(2n,\FidealI),\EU(2n,\FidealJ)\big],\EU(2n,(\FidealK)\circ(\FidealL))\Big]=\\
\Big[\EU(2n,(\FidealI)\circ (\FidealJ)),\EU(2n,(\FidealK)\circ (\FidealL))\Big]. 
\end{multline*} 
 \par
Thus, it only remains to prove that 
$$  \big[Y_{ij}(a,b),Y_{hk}(c,d)\big]\in\Big[\EU(2n,(\FidealI)\circ (\FidealJ)),\EU(2n,(\FidealK)\circ (\FidealL))\Big], $$
\noindent
where  $a\in(\FidealI)$, $b\in(\FidealJ)$, $c\in(\FidealK)$ and 
$d\in(\FidealL)$. Conjugations by elements
$x\in\EU(2n,\FormR)$ do not matter, since they amount to extra
factors from the above triple commutators, which are
already accounted for.
\par
Now, for $n\ge 4$ this already finishes the proof, since in
this case we can move $Y_{hk}(c,d)$ modulo 
$\EU(2n,(\FidealK)\circ (\FidealL))$ to a position, where it commutes with 
$Y_{ij}(a,b)]$, either by Lemma~\ref{Op-2.1} when $i\ne\pm j$ and $h\ne\pm k$  or by Lemma~\ref{Op-2.2} when $i=-j$ or $h=-k$.

Suppose that there exists an ideal equals its corresponding relative form paramerter, say $I=\Gamma$. If $i\ne \pm j$  then by Lemma~\ref{long-short}, we have 
$$Y_{i,j}(a,b)\equiv Y_{i,-i}(a, b-\lambda^{\ep(i)}\bar b).$$
For $n\ge 3$, we can move $Y_{i,-i}(a, b-\lambda^{\ep(i)}\bar b)$ module  $\EU(2n,(\FidealK)\circ (\FidealL))$ to a position, where it commutes with $Y_{hk}(c,d)$ by Lemma~\ref{Op-2.1}. Otherwise, if $i=-j$ then can also move $Y_{i,-i}(a,b)$ to a position, where it commutes with $Y_{hk}(c,d)$ by Lemma~\ref{Op-2.2}. This finishes the whole proof. 
\end{proof}


\section{Elementary multiple commutator formulas}

In the current section, we show that multiple commutators of elementary subgroups can be reduced to double such commutators.

To state our main results, we have to recall some further 
pieces of notation from \cite{yoga-1,RHZZ2,yoga-2, RNZ5, RNZ1, stepanov10}.
Namely, let $H_1,\ldots,H_m\le G$ be subgroups of $G$. There are 
many ways to form a higher commutator of these
groups, depending on where we put the brackets. Thus, for three
subgroups $F,H,K\le G$ one can form two triple commutators
$[[F,H],K]$ and $[F,[H,K]]$. Usually, we write $[H_1,H_2,\ldots,H_m]$ for the {\it left-normed\/} commutator, defined inductively by
$$ [H_1,\ldots,H_{m-1},H_m]=[[H_1,\ldots,H_{m-1}],H_m]. $$
\noindent
To stress that here we consider {\it any\/} commutator of these subgroups, with an arbitrary placement of brackets, we write $\llbracket H_1,H_2,\ldots,H_m \rrbracket$. Thus, for instance, $\llbracket F,H,K\rrbracket $ 
refers to any of the two arrangements above.
\par
Actually, a specific arrangement of brackets usually does not play
major role in our results -- apart from one important 
attribute\footnote{Actually, for non-commutative rings 
symmetric product of ideals is not associative, so that 
the initial bracketing of higher commutators will be reflected also 
in the bracketing of such higher symmetric products.}. 
Namely, what will matter a lot is the position of the outermost 
pairs of inner brackets. Namely, every higher commutator subgroup
$\llbracket H_1,H_2,\ldots,H_m\rrbracket $ can be uniquely written as
$$ \llbracket H_1,H_2,\ldots,H_m\rrbracket =
[\llbracket H_1,\ldots,H_s\rrbracket ,\llbracket H_{s+1},\ldots,H_m\rrbracket ], $$
\noindent
for some $s=1,\ldots,m-1$. This $s$ will be called the cut point
of our multiple commutator. Now we are all set to finish the
proof of Theorem~7. The proof is an easy adaptation of the proof
of \cite{NZ3}, Theorem 1, but we reproduce it here for the sake of 
completeness.

\begin{proof}
Denote the commutator on the left-hand side by $H$,
$$ H=\big\llbracket \EU(2n,I_1,\Gamma_1),\EU(2n,I_2,\Gamma_2),\ldots,\EU(2n,I_m,\Gamma_m)\big\rrbracket . $$
\noindent
We argue by induction in $m$, with the cases $m\le 4$ as the
base of induction --- for the case $m=2$ there is nothing to
prove, case $m=3$ is accounted for by Lemma~\ref{triple}, and case
$m=4$ --- by Lemma~\ref{triple}, if the cut point $s\neq 2$, and by 
Lemma~\ref{quadruple} when $s=2$.
\par
Now, let $m\ge 5$ and assume that our theorem is already 
proven for all shorter commutators. Consider an arbitrary 
arrangement of brackets\/ $[\![\ldots]\!]$ with the cut point 
$s$ and let 
\begin{multline*} 
\big\llbracket \EU(2n,I_1,\Gamma_1),\EU(2n,I_2,\Gamma_2),\ldots,
\EU(2n,I_s,\Gamma_s)\big\rrbracket ,\\
\big\llbracket \EU(2n,I_{s+1},\Gamma_{s+1}),\EU(2n,I_{s+2},\Gamma_{s+2}),\ldots,\EU(2n,I_m,\Gamma_m)\big\rrbracket , 
\end{multline*} 
\noindent
be the partial commutators, the first one containing the factors 
afore the cut point, and the second one containing those after
the cut point. 
\par\smallskip
$\bullet$ When the cut point occurs at $s=1$ or at $s=m-1$, one 
of these commutators is a single elementary subgroup $\EU(2n,I_1)$ 
in the first case or $\EU(2n,I_{m-1})$ in the second one.  Then we 
can apply the induction hypothesis to another factor. For $s=1$,
denote by $t=2,\ldots,m-1$ the cut point of the second factor.
Then by induction hypothesis
\begin{multline*}
H=\bigg[\EU(2n,I_1,\Gamma_1),\Big\llbracket \EU(2n,I_2,\Gamma_2),\EU(2n,I_3,\Gamma_3),\ldots,\EU(2n,I_m,\Gamma_m)
\Big\rrbracket\bigg]=\\
\bigg[\EU(2n,I_1,\Gamma_1),
\Big[\EU(2n,(I_2,\Gamma_2)\circ\ldots\circ(I_t,\Gamma_t)),
\EU(2n,(I_{t+1},\Gamma_{t+1})\circ\ldots\circ(I_m,\Gamma_m))\Big]\bigg], 
\end{multline*}
\noindent
and we are done by Lemma~\ref{triple}. Similarly, for $s=m-1$ denote by 
$r=1,\ldots,m-1$ the cut point of the first factor. Then by 
induction hypothesis
\begin{multline*}
H=\bigg[\Big\llbracket\EU(2n,I_1,\Gamma_1),\EU(2n,I_2,\Gamma_2),\ldots,\EU(2n,I_{m-1},\Gamma_{m-1})
\Big\rrbracket,\EU(2n,I_m,\Gamma_m)\bigg]=\\
\bigg[\Big[\EU(2n,(I_1,\Gamma_1)\circ\ldots\circ(I_r,\Gamma_r)),
\EU(2n,(I_{r+1},\Gamma_{r+1})\circ\ldots\circ(I_{m-1},\Gamma_{m-1}))\Big],\qquad\\
\hskip 4truein\EU(2n,I_m,\Gamma_m)\bigg], 
\end{multline*}
\noindent
and we are again done by Lemma~\ref{triple}. 
\par\smallskip
$\bullet$ Otherwise, when $s\neq 1,m-1$, we can apply the 
induction hypothesis to both factors. Let as above $r=1,\ldots,s-1$
be the cut point of the first factor and let $t=s+1,\ldots,m-1$ be
the cut point of the second factor. Then we can apply induction
hypothesis to both factors of
\begin{multline*}
H=\bigg[\Big\llbracket\EU(2n,I_1),\EU(2n,I_2),\ldots,\EU(2n,I_s)
\Big\rrbracket,\\
\Big\llbracket\EU(2n,I_{s+1}),\EU(2n,I_{s+2}),\ldots,\EU(2n,I_m)
\Big\rrbracket\bigg]
\end{multline*}
\noindent
to conclude that
\begin{multline*}
H=\bigg[\Big[\EU(2n,I_1\circ\ldots\circ I_r),\EU(2n,I_{r+1}\circ\ldots\circ I_s)\Big],\\
\Big[\EU(2n,I_{s+1}\circ\ldots\circ I_t),\EU(2n,I_{t+1}\circ\ldots\circ I_m)
\Big]\bigg],
\end{multline*}
\noindent
and we are again done, this time by Lemma~\ref{quadruple}.
\end{proof}


\section{Further applications}

Now, we are in a position to finish the proof of Theorem~\ref{T:8}.

\begin{proof}
Since $(\FidealI)$ and $(\FidealJ)$ are comaximal, there exist $a'\in I$ 
and $b'\in J$ such that $a'+b'=1\in R$. But then  by Lemmas~\ref{Op-2.1} 
and \ref{long-short},
for $i\ne \pm j$ one has
$$ Y_{ij}(a,b)=Y_{ij}(a(a'+b'),b)\equiv Y_{ij}(aa',b)\cdot Y_{ij}(ab',b)\equiv e$$
modulo $\EU(2n,(\FidealI)\circ (\FidealJ))$.
\par
For $i=-j$, one has 
$$
  Y_{i,-i}(a,b)=Y_{i,-i}((a'+b')a\overline{(a'+b')},b)
  =Y_{i,-i}(a'a\overline{a'}+b'a \overline{a'} +a'a\overline{b'}+b'a \overline{b'},b). $$
\noindent
Applying multiplicativity of commutators to the first argument of the above commutator and then  Lemma~\ref{Op-2}, we deduce
$$
z\equiv Y_{i,-i}(a'a\overline{a'},b)Y_{i,-i}(b'a \overline{a'},b)Y_{i,-i}(a'a\overline{b'},b)Y_{i,-i}(b'a \overline{b'} ,b)\pamod{\EU(2n,(\FidealI)\circ (\FidealJ))}.
$$
By Theorem~\ref{symb}, each of above factors is trivial modulo $\EU(2n,(\FidealI)\circ (\FidealJ))$. This finishes the proof.
\end{proof}

Let us state another amusing corollary of Theorem~\ref{symb}.
For the form ideals themselves, one has an obvious inclusion 
\begin{multline*}
\Big((\FidealI)+(\FidealJ)\Big)\circ\Big((\FidealI)\cap (\FidealJ)\Big)=\\
\Big((I+J)\circ(I\cap J), \Gamma_{\min}((I+J)\circ(I\cap J))+ {}^{(\Gamma\cap\Delta)}(\Gamma+\Delta) + {}^{(\Gamma+\Delta)}(\Gamma\cap\Delta) \Big)\le \\
\Big(I\circ J, \Gamma_{\min}(I\circ J) +{}^J\Gamma +{}^I\Delta\Big)=(\FidealI)\circ(\FidealJ).
\end{multline*}
\noindent
Only very rarely this inclusion is always an equality. 

\begin{The}\label{the:p4}
For any two form ideals $(\FidealI)$ and $(\FidealJ)$ of $(\FormR)$, $n\ge 3$, one has
$$ \Big[\EU\big(2n,(\FidealI)+(\FidealJ)\big),
\EU\big(n,(\FidealI)\cap(\FidealJ)\big)\Big]\le 
\big[\EU(2n,\FidealI),\EU(2n,\FidealJ)\big]. $$
\end{The}

\begin{proof}
The observation immediately preceding the theorem shows that
the level of the left hand side is contained in the level of the right
hand side,
$$ \EU\Big(2n,R,
\big((\FidealI)+(\FidealJ)\big)\circ
\big((\FidealI)\cap(\FidealJ)\big)\Big)\le 
\EU\big(2n,R,(\FidealI)\circ(\FidealJ)\big). $$
\par
Thus, it only remains to prove that the elementary commutators
$Y_{ij}(a+b,c)$, with $a\in(\FidealI)$, $b\in(\FidealJ)$, 
$c\in(\FidealI)\cap(\FidealJ)$,
in the left hand side belong to the right hand side.
\par
By Theorem~\ref{symb}, one has
$$ Y_{ij}(a+b,c)\equiv Y_{ij}(a,c)\cdot Y_{ij}(b,c)
\pamod{\EU\big(2n,R,
((\FidealI)+(\FidealJ))\circ((\FidealI)\cap(\FidealJ))\big)}. $$
\noindent
Thus, this congruence holds also modulo the larger subgroup
$\EU(2n,R,(\FidealI)\circ(\FidealJ))$. 
\par
On the other hand, Theorem~\ref{The:Op-2.1} implies that 
$$ Y_{ij}(b,c)\equiv Y_{ij}(c,-b)
\pamod{\EU(2n,R,(\FidealI)\circ(\FidealJ))}. $$ 
\par
Combining the above congruences, we see that
$$ Y_{ij}(a+b,c)\equiv  Y_{ij}(a,c)\cdot Y_{ij}(c,-b)
\pamod{\EU(2n,R,(\FidealI)\circ(\FidealJ))}, $$ 
\noindent
where both commutators in the right hand side belong to
$[\EU(2n,\FidealI),\EU(2n,\FidealJ)]$, which proves the desired inclusion.
\end{proof}


\section{Final remarks}

Here we make some further observations concerning the context 
of this work and also state some unsolved problems and reiterate 
some further problems from \cite{RNZ1, RNZ5}, which are still 
pending.  

\subsection{How we got here.} The study of birelative standard
commutator formulas goes back to the foundational work by 
Hyman Bass \cite{Bass_stable}. As early successes one should 
also mention important contributions by Alec Mason and Wilson 
Stothers \cite{MAS3, Mason74, MAS1, MAS2} and by Hong You \cite{HongYou}.
Our own research in this direction started in 2008--2010 in the joint
works with Alexei Stepanov and Roozbeh Hazrat \cite{NVAS, RHZZ1, NVAS2} and was then continued in 2011--2017 in a series of our 
joint works based on relative versions of localisation 
methods, 
in particular\footnote{At least three our scheduled works of that 
period, which were essentially completed by 2016, viz., the
general multiple commutator formula for $\GL(n,R)$, unitary 
commutator width, and analysis of the case $\GU(4,R,\Lambda)$, 
still remain
unpublished.} \cite{ RHZZ2, RNZ1, RNZ2, RNZ3, RNZ4, RNZ5}. 
Simultaneously, Stepanov developed his universal localisation
and applied it to multiple commutator formulas and commutator 
width, see \cite{Stepanov_nonabelian, stepanov10}. One can find
systematic description of that stage of development in our
surveys and conference papers \cite{yoga-1, portocesareo, yoga-2, RNZ4}. 

The present work is a natural extension of our more recent papers
\cite{NV18, NZ2, NV19, NZ1, NZ3, NZ6, NZ4}. It owes its 
existence to the two following momentous observations we 
made in October 2018, and in September 2019, respectively.
\par
In October 2018 the first author proved a special case of 
Theorems~\ref{equality2} and \ref{unrelative} for the general linear group $\GL(n,R)$,
$n\ge 3$, over commutative rings, see \cite{NV18}. The
initial proof employed a version of decomposition of unipotents 
\cite{ASNV}, that was already used for a similar purpose
in his joint work with Alexei Stepanov \cite{NVAS}. The second 
author then immediately observed that Theorem~\ref{equality2} implies the
first claim of Theorem~\ref{generators} and that it should be possible to proceed conversely, first establish a version of Theorem~\ref{generators} by elementary calculations, and then derive Theorems~\ref{equality2} and \ref{unrelative}. This is exactly 
what was done for Chevalley groups in our paper \cite{NZ2}, 
again over commutative rings.
\par
In July--September 2019 the first author was discussing bounded
generation of Chevalley groups in the function case with Boris Kunyavsky and Eugene Plotkin. One of
the tricks used in many published papers consisted in splitting
an elementary conjugate/elementary commutator and then reassembling it in a different position. We noticed that the same calculation of rolling
elementary conjugates to a different position appeared over
and over again in many different contexts:
\par\smallskip
$\bullet$ {\it Congruence subgroup problem.\/} In a preliminary mode
it was already present in the precursory article by Jens Mennicke
\cite{Mennicke}
and then already in full-fledged form in the epoch-making memoir
by Hyman Bass, John Milnor, and Jean-Pierre Serre
\cite{Bass_Milnor_Serre}, behold the proof of 
Theorem 5.4. 
\par\smallskip
$\bullet$ {\it  Bounded generation\/}. Post factum, we discerned the
same calculation in the classical papers by David Carter, Gordon 
Keller, and Oleg Tavgen \cite{CK83, Tavgen1990}, but we only 
became aware of that perusing a recent article by Bogdan Nica \cite{Nica}.
\par\smallskip
$\bullet$ In fact, Wilberd van der Kallen and Alexei Stepanov 
\cite{vdK-group, Stepanov_calculus, Stepanov_nonabelian} use 
a very similar calculation to reduce the generating sets of 
relative elementary subgroups.
\par\smallskip\noindent
Here we attached merely a handful of references. Retrospectively, 
we spotted the same or very similar calculations in oodles of further
papers, but apparently it was hardly ever applied in the birelative context.
\par
At the end of September the first author used essentially the same 
calculation\footnote{Simultaneously and independently exactly 
the same calculation was applied by Andrei Lavrenov and Sergei 
Sinchuk \cite{Lavrenov_Sinchuk} at 
the level of $\K_2$.} to prove that when $R$ is commutative and 
$n\ge 3$ the mixed relative commutator subgroup $[E(n,A),E(n,B)]$ is contained in another birelative group
$$ \EE(n,A,B)=\big\langle t_{ij}(c),\text{\ where\ } 
c\in A, i<j,\text{\ and\ } c\in B, i>j\big\rangle, $$
\noindent
see \cite{NV19}, Theorem 3. Within a few days of vehement correspondence we observed that everything works over arbitrary associative rings and can be further enhanced to entail 
Theorems 1 and 5 for $\GL(n,R)$. This is done in \cite{NZ2}, and 
soon thereafter in a more mature form, implying also Theorems 6,
7 and 8, in \cite{NZ3}.
\par
Morally, the present paper, and a parallel paper that addresses 
the case of Chevalley groups \cite{NZ4}, are direct offsprings
of this development. However, technically these cases turned 
out to be way more demanding, and we had to spend quite some 
time to supply detailed proofs of all auxiliary results.

\subsection{Degree improvements.}
Of course, the first question that immediately occurs is 
whether Theorem~\ref{T:7} holds also for $n=3$. For 
{\it quasi-finite\/} rings this is indeed the case \cite{RNZ4},
and we are pretty more inclined to believe in the positive answer.

\begin{Prob}
Prove that Lemma~$\ref{quadruple}$ and Theorem~$\ref{T:7}$
hold also for $n=3$.
\end{Prob}

Getting a proof in the same style as that of Lemma~\ref{triple}
seems to be highly non-trivial from a technical viewpoint. However, 
the possibility to construct a counter-example appears even 
more remote.

In the main body of the present paper we always assumed that 
$n\ge 3$. Obviously, due to the exceptional behavior of the 
orthogonal group $\SO(4,A)$, these results do not fully generalise 
to the case $n = 2$. It is natural to ask, whether results 
of the present paper hold also for the group $\GU(4,\FormR)$. 
However, this obviously fails in general without some strong
additional assumptions on the form ring and/or form ideals. 

Still, we believe they do generalise, provided 
$\Lambda A+A\Lambda=A$, or the like. Known 
results\footnote{Compare the work by Bak and the first author \cite{BV1}, and references therein.} clearly indicate both that this should be 
possible, and that the analysis of the case $n = 2$ will be 
considerably harder from a technical viewpoint, than that 
of the case $n\ge 3$.

\begin{Prob}\label{p4} 
Generalise results of the present paper to the group
$\GU(4,A,\Lambda)$, provided that
$\Lambda A+A\Lambda=A$, $\Gamma J+J\Gamma =I$,
$\Delta I+I\Delta=J$, or the like.
\end{Prob}

Actually, some 8 years ago we have obtained various headways 
towards the relative standard commutator formula and all that for 
$\GU(4,A,\Lambda)$, but even these results are unpublished,
due to their fiercely technical character.

\subsection{Presentations and stability.}
As a counterpart to Theorem~\ref{T:9} we can ask, whether the 
stability map for this quotient is also injective. A natural approach 
to this would be to tackle the following much more ambitious 
project.

\begin{Prob}
Give a presentation of
$$ \big[\EU(2n,\FidealI),\EU(2n,\FidealJ)\big]
/\EU(2n,A,(\FidealI)\circ(\FidealJ)) $$
\noindent
by generators and relations. Does this presentation depend on 
$n\ge 3$? 
\end{Prob}

In Theorems~\ref{The:Op-2.1} and \ref{symb} and Lemma~\ref{long-short} we have established some 
of the relations among the elementary commutators modulo
$\EU(2n,A,(\FidealI)\circ(\FidealJ))$. However, easy arithmetic
examples show this is not a defining set of relations, so that there 
must be some further relations. Compare \cite{NZ2, NZ3, NZ6} for 
discussion of the similar problem for $\GL(n,A)$.

\subsection{Higher relations.} In \cite{NZ6} we established some 
further congruences for the elementary commutators in
$\GL(n,A)$, $n\ge 3$, where $A$ is an arbitrary associative ring.
The highlight of that paper is the following remarkable triple 
congruence, a version of the Hall---Witt identity.

Let $I,J,K$ 
be two-sided ideals of $R$. Then for any  three distinct indices $i,j,h$
such that $1\le i,j,h\le n$, and all $a\in I$, $b\in J$, $c\in K$, 
one has
$$ y_{ij}(ab,c) y_{jh}(ca,b) y_{hi}(bc,a)\equiv e
\pamod{E(n,R,IJK+JKI+KIJ)}, $$
\noindent
see \cite{NZ6}, Theorem 1. This identity has lots of applications,
including many new inclusions among double and multiple mixed 
relative elementary commutator subgroups.
\par
Specifically, it allows to solve the analogue of Problem 3 for 
$\GL(n,A)$ in the particularly agreeable case of Dedekind rings. 
Thus, it would be most natural to seek out similar higher 
congruences in the unitary case as well.

\begin{Prob}
Generalise the results of \cite{NZ6} to the unitary groups
$\GU(2n,\FormR)$, $n\ge 3$.
\end{Prob}

One such congruence among {\it short\/} root type elementary commutators is immediately clear. But the congruences involving
long root type elementary commutators will be fancier and longer.

\subsection{Other birelative groups.}
Let us briefly discuss two further groups depending on two form 
ideals of a form ring. First of all, it is the partially relativised 
group ${\FU(2n,\FidealI)}^{\FU(2n,\FidealJ)}$.
It seems that in view of the identity
$$ {\FU(2n,\FidealI)}^{\FU(2n,\FidealJ)}=
[\FU(2n,\FidealI),\FU(2n,\FidealJ)]\cdot \FU(2n,\FidealI), $$
\noindent 
our Theorem~\ref{generators} readily implies the following generalisation of 
\cite{BV3}, Proposition 5.1, to 
${\FU(2n,\FidealI)}^{\FU(2n,\FidealJ)}$.
Namely, we assert that it is generated by the appropriate 
elementary conjugates.

\begin{Prob}\label{p7.1} 
Prove that the partially relativised groups 
${\FU(2n,\FidealI)}^{\FU(2n,\FidealJ)}$ are generated by
${}^{T_{ji}(b)}T_{ij}(a)$, where $a\in(\FidealI)$, $b\in(\FidealJ)$.
\end{Prob}

Another birelative group $\EEU(2n,(\FidealI),(\FidealJ))$ is
defined as follows
$$ \EEU(2n,(\FidealI),(\FidealJ))=
\big\langle T_{ij}(a),\text{\ where\ } c\in(\FidealI), i<j,\text{\ and\ } 
c\in(\FidealJ), i>j\big\rangle. $$
\par
The following problem proposes a unitary generalisation
of \cite{NV19}, Theorem 3, where a similar result was
established for $\GL(n,A)$.

\begin{Prob}\label{p6} 
Prove that 
$$ [\FU(2n,\FidealI),\FU(2n,\FidealJ)]\le
 \EEU(2n,(\FidealI),(\FidealJ)). $$
\end{Prob}

\subsection{General multiple commutator formula.}
Let us now recall another major unsolved problem as stated 
already in \cite{RNZ1, RNZ4} and \cite{RNZ5}, Problem 1.
We proffer to prove general multiple commutator formula 
for unitary groups.

\begin{Prob}\label{p9} 
Let $(I_i,\Gamma_i)$, $1\le i\le m$, be form ideals of the form
ring $(A,\Lambda)$ such that $A$ is module-finite over a
commutative ring $R$ that has finite Bass–Serre dimension 
$\delta(R)=d<\infty$. Prove that for any $m\ge d$ one has
\begin{multline*}
\big\llbracket\GU(2n,I_0,\Gamma_0),\GU(2n, I_1,\Gamma_1),\ldots,
\GU(2n,I_m,\Gamma_m)\big\rrbracket=\\
\big\llbracket\EU(2n,I_0,\Gamma_0),\EU(2n, I_1,\Gamma_1),\ldots,
\EU(2n, I_m,\Gamma_m)\big\rrbracket. 
\end{multline*}
\end{Prob}

Observe that the arrangement of brackets in the above formula 
should be the same on both sides as the mixed commutators are
not associative. A similar problem for algebraic groups over {\it commutative\/} rings, in particular for Chevalley groups, was solved
by Alexei Stepanov \cite{stepanov10}, by his remarkable
{\it universal localisation\/} method.
\par
Recall that the proof of a similar result for $\GL(n,R)$ over
{\it non-commutative\/} rings is based on the following result of Mason---Stothers \cite{MAS3}, Theorem 3.6 and Corollary 3.9, 
see \cite{RNZ4}, Theorem 13, for an easy modern proof. Of
course, that we can unrelativise the right hand side was only 
established in \cite{NZ2}, Theorem 2, so formally this
theorem was never stated in this form.

\begin{oldtheorem}
Let $A$ be a ring, $I$ and $J$ be two two-sided
ideals of $A$. Assume that $n\ge\sr(R),3$. Then
$$ [\GL(n,A,I),\GL(n,A,J)] = [E(n,I),E(n,J)]. $$
\end{oldtheorem}

For unitary groups, even such basic facts at the stable level 
seem to be missing. 

\begin{Prob}\label{p8} 
Find appropriate stability conditions under which
$$ [\GU(2n,\FidealI),\GU(2n,\FidealJ)]=
[\FU(2n,\FidealI),\FU(2n,\FidealJ)]. $$
\end{Prob}

After that, the proof in our unpublished paper proceeds 
by induction on $d$, which depends on Bak’s results \cite{Bak}, 
precise form of injective stability for $K_1$, such as the 
Bass–Vaserstein theorem, etc. It seems that to solve Problem 7 
one has to rethink and expand many aspects of structure theory 
of unitary groups, starting with stability theorems for $\KU_1$.

The first complete\footnote{In late 1960-ies
and mid 1970-ies Anthony Bak 
and  Manfred Kolster obtained stability under stronger assumptions,
with very sketchy proofs. Leonid Vaserstein worked in smaller 
generality as far as groups, and his proof of injective stability for
unitary groups 
contained serious gaps and inaccuracies. In 1980 Mamed-Emin Oglu Namik Mustafa-Zadeh announced surjective stability for $\KU_2$ ---
and thus also injective stability for $\KU_1$ ---
in full generality. However, a complete proof was never published,
and the exposition in his 1983 Ph.~D.~Thesis is blurred by serious mistakes.}  generally accepted proof of injective stability 
for $\KU_1$ was obtained (but not published!)~by Maria
Saliani \cite{Saliani}, and first published by Max Knus in his book 
\cite{knus}. After that, generalisations and improvements were
proposed by Anthony Bak, Guoping Tang, Victor Petrov,
and Sergei Sinchuk \cite{BT, BPT, Sinchuk}, and then very
recently by Weibo Yu, Rabeya Basu and Egor Voronetsky 
\cite{Yu, Basu18, Voron-2}.
\par
Problem 7 is also intimately related to the nilpotent structure 
of $\KU_1$. In the absolute case the corresponding results for
unitary groups were obtained by Roozbeh Hazrat in his 
Ph.~D.~Thesis \cite{RH, RH2}, and in the
relative case in a joint paper by Bak, Hazrat and the first 
author \cite{BHV}. To fully cope with Problem 7, we need more 
powerful results on the superspecial unitary groups than what 
was established in \cite{BHV}. Part of what is demanded here 
was recently established by Weibo Yu, Guoping Tang and Rabeya 
Basu \cite{Yu-Tang, Basu16}, but there is still a lot of work to be
done.

\subsection{Subnormal subgroups.}
Initially, one of our main motivations to pursue the work 
on birelative commutator formulas were prospective 
applications to the study of subnormal subgroups of 
$\GU(2n,A,\Lambda)$. As was observed by John Wilson
\cite{Wilson}, technically this amounts
to description of subgroups of $\GU(2n,\FormR)$, 
normalised by a relative elementary subgroup $\EU(2n,\FidealJ)$,
for {\it some\/} form ideal $(\FidealJ)$.
\par
A major early contribution is due to Günter Habdank \cite{Ha1, Ha2},
who additionally assumed that the form ring was subject to some stability conditions.
Definitive results for quasi-finite rings were then obtained by the second author and You Hong \cite{ZZ, ZZ1, ZZ2, You_subnormal}. 
However, we are convinced that the bounds in these papers can be further improved and hope to return to the following problem with 
our new tools.

\begin{Prob} Obtain optimal bounds in the description of subgroups 
of $\GU(2n,\FormR)$, normalised by the relative elementary
subgroup $\EU(2n,\FidealJ)$, for a form ideal 
$(\FidealJ)\unlhd(\FormR)$.
\end{Prob}

Until recently, for the unitary groups the proofs of structure theorems
were in bad shape even in the absolute case\footnote{As indicated in \cite{RN}, the proof in the work by Leonid Vaserstein and Hong You \cite{VY} contained a major omission, and only established the 
{\it weak\/} structure theorem. The details of the purported global 
proof by Bak and the first author, that was around since the early 
1990-ies, and that was harbingered in \cite{BV3}, remained unpublished.}. However, now the situation has changed. In 2013 
Hong You and 
Xuemei Zhou \cite{You-Zhou} published a detailed proof for
commutative form rings. Finally, in 2014 Raimund Preusser in
his Ph.~D.~Thesis \cite{Preusser-thesis} gave a first complete
{\it localisation proof\/} for quasi-finite form rings, which is 
published in \cite{preusser-1}.
\par
In 2017 Raimund Preusser \cite{preusser-2, preusser-3} has also 
finally succeeded in completing a {\it global proof\/} as
envisaged in \cite{BV3}. These papers constitute a major 
breakthrough since, at least for commutative rings, they give 
explicit polynomial expressions of non-trivial transvections as 
products of elementary conjugates of a given matrix and its inverse.
(See also \cite{preusser-4, preusser-6} for further results in this
spirit for $\GL(n,A)$ over various classes of non-commutative 
rings.)
The first author has immediately recognised  that the results
by Preusser procure an effectivisation for the description of
normal subgroups in much the same sense as the {\it 
decomposition of unipotents\/} \cite{ASNV}, does for the 
normality of the elementary subgroup. This prompted him
to call this method {\it reverse decomposition of unipotents}
\cite{NV-reverse}. Moreover, he noticed that in the case
of $\GL(n,A)$ these results can be generalised (with only 
marginally worse bounds) to the description of subgroups
normalised by a {\it relative\/} elementary subgroups
\cite{NV-reverse-2}.
\par
We are confident that, combining the methods developed by
Preusser in the above papers with our methods, we could
easily improve bounds in all published results for unitary
groups. Of course, to prove that the bounds thus obtained 
are themselves the best possible ones would be quite a 
challenge.

\subsection{Commutator width.} 
Another related problem that initially motivated our work 
was the study of commutator width. 
Alexander Sivatsky and Alexei Stepanov 
\cite{SiSt} have discovered that over rings of finite Jacobson 
dimension $\jdim(A)=d<\infty$ any commutator $[x,y]$, 
where $x\in\GL(n,A)$, $y\in E(n,A)$, is a product of $\le L$
elementary generators, where $L=L(n,d)$ only depends on
$n$ and $d$. This result was then generalised to all Chevalley 
groups $G(\Phi,A)$ by Stepanov and the first author \cite{SV11},
with the bound depending on the type $\Phi$ and on $d$.
\par
Ultimately, Stepanov discovered that for {\it reductive groups} 
similar results hold for {\it arbitrary\/} commutative rings and
that the bound $L$ therein depends on the type of the group 
alone and not on the ring $A$. Also,
he discovered that similar results hold at the relative and birelative 
level, with elementary conjugates and our generators (like those
in Theorem B) as the
generating sets of $[E(\Phi,A,I),E(\Phi,A,J)]$, again with bounds
that depend on the type alone, and not on $A$, $I$ or $J$. See \cite{portocesareo}
for statements and detailed discussion of these results.
\par
However, Bak's unitary groups are not always algebraic and similar
results on commutator width are not yet published even in
the absolute case and even over finite-dimensional rings.

\begin{Prob}
Let $(\FormR)$ be a commutative form ring such that
$\jdim(A)<\infty$. Prove that the length of commutators in 
$[\GU(\Phi,A,I),E(\Phi,A,J)]$ in terms of the generators listed in
Theorem~$\ref{generators}$ is bounded, and estimate this length.
\end{Prob}

Alexei Stepanov maintained that the above length is bounded in 
the absolute case, without actually producing any specific bound. 
To obtain an exponential bound depending on $d$ by relative
localisation methods \cite{RNZ1, RNZ5, RNZ4} would be simply
a matter of patience. Actually, this was essentially done by ourselves
and Roozbeh Hazrat, but even 
in the absolute case all of this still remains unpublished.
\par
On the other hand, to achieve a {\it uniform\/} polynomial bound, 
similar to the one established in \cite{SiSt} for $\GL(n,A)$ but not
depending on $d$, one would need to combine a full-scale 
generalisation of Stepanov's universal localisation to unitary 
groups, with full-scale unitary versions of decomposition 
of unipotents, including explicit polynomial formulae 
for the conjugates of root unipotents. This seems to be a rather
ambitious project.

\subsection{Unitary Steinberg groups.} 
It is natural to ask to which extent our methods and results
carry over to the level of $\KU_2$.

\begin{Prob} 
Prove analogues of the main results of the present paper
for the unitary Steinberg groups $\StU(2n,\FormR)$.
\end{Prob}

For the definition of unitary Steinberg groups see
\cite{B2, lavrenov} and references there (or \cite{lavrenov-bis}
for odd unitary Steinberg groups). Here, we do not discuss 
subtleties related to the definition of relative unitary Steinberg 
groups, as also relation to excision in unitary algebraic $K$-theory, 
etc.

\subsection{Description of subgroups.} The methods of the 
present paper can have applications also in description of 
various classes of subgroups of unitary groups. Not in the
position to discuss this at any depth here, we just cite
the works by Victor Petrov, Alexander Shchegolev and 
Egor Voronetsky \cite{petrov1, Shchegolev-thesis, Shchegolev-1,
Shchegolev-2, Voron-1} where one can find many further 
references. Observe that the result by Voronetsky \cite{Voron-1} 
is especially powerful, since it simultaneously generalises also
the description of $\EU$-normalised subgroups (in the context
of odd unitary groups!)

\subsection{Odd unitary groups.}
Finally, we are positive that all results of the present paper 
generalise also to odd unitary groups introduced by 
Victor Petrov \cite{petrov2, petrov3}. 

\begin{Prob}\label{p7} 
Generalise the results of \cite{RNZ1, RNZ3, RNZ4} and the 
present paper to odd unitary groups, under suitable isotropy
assumptions.
\end{Prob}

Of course, this is not an individual clear-cut problem, but 
rather a huge research project. Clearly, in most cases the 
proofs in this setting will require much more onerous calculations.
Let us cite some important recent papers by Yu Weibo, 
Tang Guoping, Li Yaya, Liu Hang, Anthony Bak, Raimund Preusser
and Egor Voronetsky \cite{Yu-Tang, Yu_Li_Liu, BP, preusser-5,
Voron-1, Voron-2} that address normal structure and stability 
for odd unitary groups.

\subsection{Acknowledgements.}
We thank Anthony Bak, Roozbeh Hazrat and Alexei Stepanov 
for long-standing close cooperation on this type of problems 
over the last decades. The present paper gradualy evolved to 
the current shape between December 2018 and March 2020.
The first author thanks Boris Kunyavsky and Eugene Plotkin, 
for ongoing 
discussion and comparison of the 
existing proofs of the congruence subgroup problem and 
bounded generation in terms of elementaries. The bout of
these deliberations that has taken place on September 16, 
2019, first in ``Biblioteka Cafe'', and then in ``Manneken Pis'' 
on Kazanskaya, was especially fateful for \cite{NV19} and all subsequent development. We thank Pavel Gvozdevsky,
Andrei Lavrenov, Sergei Sinchuk and Anastasia Stavrova
for their very pertinent questions and comments.
We are extremely grateful also to Fan Huijun for his 
friendly support. In particular, he organised a visit of the first 
author to Peking University in December 2019, which gave us
an excellent opportunity to coordinate our vision. 


\end{document}